\documentclass{amsproc}

\usepackage{amssymb}

\usepackage{amscd,mathtools}
\usepackage{latexsym}
\usepackage{comment}

\usepackage{amssymb}
\usepackage{amsfonts}
\usepackage{graphicx}
\usepackage{hyperref}

\newcommand{\Z}{{\mathbb Z}}
\newcommand{\isi}{{\mathcal I}}
\newcommand{\isd}{{\mathcal D}}
\newcommand{\Q}{{\mathbb Q}}
\newcommand{\LL}{{\mathcal{L}}}
\newcommand{\R}{{\mathbb R}}

\newcommand{\HH}{\mathcal{H}}

\newcommand{\D}{\mathcal{D}}

\newcommand{\A}{\mathbb{A}}

\newcommand{\PP}{\mathcal{P}}
\newcommand{\RR}{\mathcal{R}}
\newcommand{\qz}[1]{\href{https://www.lmfdb.org/NumberField/3.3.\ifnum#1=7 49\else 81\fi.1}{{\Q(\zeta_{#1})}^+}}
\renewcommand{\P}{\mathbb{P}}
\newcommand\N{{\mathbb N}}
\renewcommand\O{{\mathcal O}}
\newcommand{\Sy}{\mathcal{S}}

\newcommand{\p}{{\mathfrak{p}}}
\newcommand{\frP}{{\mathfrak{P}}}
\newcommand{\fra}{{\mathfrak{a}}}
\newcommand{\dif}{{\mathfrak{d}}}
\newcommand{\rr}{{\mathfrak{r}}}
\newcommand{\otp}{{\O^\times_{K,+}}}
\newcommand{\kpow}[2]{{K_{#1}^{\otimes #2}}}
\DeclareMathOperator{\Cl}{Cl}
\newcommand\disc{D}
\DeclareMathOperator{\Aut}{Aut}

\DeclareMathOperator{\vol}{vol}

\newcommand{\iu}{{{_i}u}}
\newcommand{\iv}{{{_i}v}}
\DeclareMathOperator{\HHH}{H}
\DeclareMathOperator{\hhh}{h}

\DeclareMathOperator{\tr}{tr}

\DeclareMathOperator{\covol}{covol}
\DeclareMathOperator{\GL}{GL}
\DeclareMathOperator{\SL}{SL}
\DeclareMathOperator{\PGL}{PGL}
\DeclareMathOperator{\PSL}{PSL}
\DeclareMathOperator{\coker}{coker}
\DeclareMathOperator{\im}{im}

\newcommand{\nflmfdb}[4]{\href{https://www.lmfdb.org/NumberField/#1.#2.#3.#4}{#1.#2.#3.#4}}
\newcommand{\cf}[1]{\nflmfdb{3}{3}{#1}{1}}

\theoremstyle{plain}
\newtheorem{thm}{Theorem}[section]

\newtheorem{hyp}[thm]{Hypothesis}

\newtheorem{alg}[thm]{Algorithm}
\newtheorem{prop}[thm]{Proposition}
\newtheorem{cor}[thm]{Corollary}
\newtheorem{lem}[thm]{Lemma}
\theoremstyle{definition}
\newtheorem{ex}[thm]{Example}

\newtheorem{notation}[thm]{Notation}
\newtheorem{myremark}[thm]{Remark}
\newtheorem{problem}[thm]{Problem}
\newtheorem{defn}[thm]{Definition}
\newtheorem{question}[thm]{Question}	%\renewcommand{\thequestion}{}
\newenvironment{copythm}[1]{\def\copiedthm{{#1}}\begin{thethmcopy}}{\end{thethmcopy}}
\newtheorem*{thethmcopy}{Theorem \copiedthm}
\newenvironment{copyprop}[1]{\def\copiedprop{{#1}}\begin{thepropcopy}}{\end{thepropcopy}}
\newtheorem*{thepropcopy}{Proposition \copiedprop}
\theoremstyle{remark}

\newtheorem{theorem}{Theorem}[section]

\theoremstyle{definition}

\theoremstyle{remark}

\numberwithin{equation}{section}

\begin{document}

%\contrib[with an appendix by]{Adam Logan and John Voight}
\title{The Kodaira dimension of Hilbert modular threefolds}

%    author one information
\author{Adam Logan}
\address{The Tutte Institute for Mathematics and Computation,
  P.O. Box 9703, Terminal, Ottawa, ON K1G 3Z4, Canada}
\curraddr{School of Mathematics and Statistics, 4302 Herzberg Laboratories, 1125 Colonel By Drive, Carleton University, Ottawa, ON K1S 5B6, Canada}
\email{adam.m.logan@gmail.com}
\thanks{}

%    The 2020 edition of the Mathematics Subject Classification is
%    the current definitive version.

\date{April 17, 2026}

\begin{abstract}
  Following a method introduced by Thomas-Vasquez and developed by Grundman,
  we prove that many Hilbert modular threefolds of geometric
  genus $0$ and $1$ are of general type, and that some are of nonnegative
  Kodaira dimension.  The new ingredient is a detailed study
  of the geometry and combinatorics of totally positive integral elements
  $x$ of a fractional ideal $I$ in a totally real number field $K$ with
  the property that ${\mathop{\rm tr}} xy < \min I {\mathop{\rm tr}} y$
  for some $y \gg 0 \in K$.
\end{abstract}

\maketitle

\section{Introduction}
This paper is about the birational geometry, and in particular the Kodaira
dimension, of Hilbert modular varieties.  We begin by recalling
the definition of these varieties as complex orbifolds, as in
\cite[Chapter 1]{vdg}, \cite[\S2.2]{freitag}.
Let $K$ be a totally real number field of degree $d$.  The standard
formula $\begin{psmallmatrix}a&b\\c&d\end{psmallmatrix} z = \frac{az+b}{cz+d}$
gives an action of $\GL_2^+(\R)$ on the upper half-plane $\HH$.  The
product of the $d$ embeddings $K \hookrightarrow \R$ gives an injection
$\PSL_2(\O_K) \hookrightarrow \GL_2^+(\R)^d$ whose image is discrete and acts
on $\HH^d$ properly discontinuously and with finite stabilizers; thus the
quotient inherits a structure of complex orbifold from $\HH^d$, which is the
most basic Hilbert modular variety.

\begin{myremark}\label{rem:hmv-components}
  It is sometimes important to consider a quotient by
  $\PGL_2^+(\O_K)$, the subgroup of $\GL_2(\O_K)$ consisting of matrices whose
  determinant is a totally positive unit,
  rather than $\PSL_2$.  These groups are the same when the class
  number and narrow class number of $K$ are equal, but if $\varepsilon$ is a
  totally positive unit that is not a square
  then $\begin{psmallmatrix}\varepsilon&0\\0&1\end{psmallmatrix}$
    %the diagonal matrix with entries $\varepsilon,1$
    represents an element of $\PGL_2^+$ not in $\PSL_2$.  Again the quotient
    is a complex orbifold of dimension $d$.
    % need to be careful about GL_+ vs. SL
\end{myremark}

Let $A$ be a fractional ideal of $\O_K$ and consider
$\Aut(\O_K \oplus A)$.  This can be viewed as a group of matrices
$\begin{psmallmatrix}a&b\\c&d\end{psmallmatrix}$, where $a, d \in \O_K$,
$b \in A^{-1}$, $c \in A$ and where the determinant belongs to $\O_K^\times$.
We thus write $\PSL(\O_K \oplus A), \PGL^+(\O_K \oplus A)$ for the subgroups of
$\Aut(\O_K \oplus A)$ for which the determinant is the square of a unit,
respectively a totally positive unit, modulo the scalar matrices.

We state some results from \cite[I.4]{vdg} on these groups.
If $A$ is generated by a totally positive element $x_+ \in K$, then
$\PSL(\O_K \oplus A)$ is conjugate to $\PSL_2$ by 
$\begin{psmallmatrix}x_+&0\\0&1\end{psmallmatrix}$, which identifies the
action of $\PSL_2$ on $\HH^d$ with that of $\PSL(\O_K \oplus A)$.
Similarly, if $A = I^2$, then the two groups are conjugate by a matrix
of determinant $1$ whose two columns belong to $I, I^{-1}$.  On the other
hand, the groups $\PSL(\O_K \oplus A), \PSL(\O_K \oplus B)$ are not conjugate in
$\GL_2^+(K)$ if $AB^{-1}$ is not of the form $TI^2$, where
$T$ is a narrowly principal ideal.  Thus the conjugacy classes of groups of
this form are in bijection with the genera of $K$.  Similarly for
$\PGL^+$.

\begin{defn}\label{def:gamma-0} \cite[(2.1.4)]{hmf-mult}
  Let $I$ be a nonzero ideal of $\O_K$ and let $A$ be a fractional ideal.
  Let $\Gamma_0(I;A), {\hat \Gamma}_0(I;A)$ be the subgroups of
  $\PSL(\O_K \oplus A), \PGL^+(\O_K \oplus A)$ respectively consisting of
  matrices whose lower left entry belongs to $AI$.  These are subgroups of
  finite index, so we again obtain a complex orbifold as the quotient.  This
  parametrizes abelian $d$-folds with a particular kind of action of $\O_K$
  and a distinguished $\O_K$-submodule of the torsion isomorphic to $\O_K/I$.
\end{defn}

\begin{defn}\label{def:hmv}
  Let $K$ be a totally real number field and $\Gamma$ a subgroup of finite index
  of $\PGL_2^+(\O_K)$ as above.  The quotient $\Gamma \backslash \HH^d$ is the
  {\em Hilbert modular variety} for the group $\Gamma$.
\end{defn}

A priori this is only a complex orbifold, but in fact it has an algebraic
structure by \cite[Theorem II.7.1]{vdg}.

\begin{notation}\label{not:hmv}
  The Hilbert modular variety for the group $\Gamma_0(I;A)$, where $I$ is an
  ideal in $\O_K$, will be denoted $H_{K,I;A}$.  When $A$ is $(1)$ we will
  generally omit it from the notation, writing $H_{K,I}$.  Similarly
  if $I = (1)$ we may write $H_{K;A}$, and if both are $(1)$ we simply
  write $H_K$.
\end{notation}

These quotients are not compact; they are compactified by adjoining finitely
many points called cusps.  Since our main concern is with
birational geometry, we will not always distinguish between the open and
compactified varieties, referring to both as Hilbert modular varieties.
The cusps of $H_K$ are in bijection with
$\PSL_2(\O_K) \backslash \P^1(K)$, or equivalently with
ideal classes of $\O_K$.  Giving a resolution of the cusps
is essentially equivalent to decomposing a fundamental domain for the action
of $\otp$, the totally positive units of $\O_K$, on the totally
positive cone in $K \otimes \R$, into cones spanned by bases for ideals.
In the case $[K:\Q] = 2$ this is treated in great detail in 
\cite[Chapter 2]{vdg}; in \cite{tv-rcs} the problem is
discussed briefly for arbitrary degree in Section 1 and studied much
more closely in the rest of the paper.
In contrast to the case
$[K:\Q] = 2$ where there is a canonical way to subdivide a cone in $\R^2$,
there are many different ways to resolve the cusp singularities in
dimension greater than~$2$, none of which is obviously the most natural.
As pointed out in \cite[p.~36]{vdg}, this is related to the lack of a single
minimal model for varieties of dimension greater than $2$; for discussion
of a related situation the reader may wish to consult
\cite[Chapter 14]{matsuki}.

In addition, there are singularities that arise from the isolated fixed points
of elements of $\Gamma$
on $\HH^d$.  For $[K:\Q]$ fixed there are only finitely many such
types and they are much more tractable than the cusps.

\subsection{Previous results}
Recall the definitions of Kodaira dimension $\kappa_V$
and general type \cite[Definition~1-5-2, Proposition~1-5-3]{matsuki}:
\begin{notation} Let $V$ be a smooth variety.  The canonical line bundle
  of $V$ will be denoted $K_V$; since we are not concerned with specific
  divisors $D$ with $\O(D) \cong K_V$, we write its powers as $\kpow{V}{n}$
  rather than $nK_V$.
\end{notation}

\begin{defn} Let $V$ be an irreducible
  variety and $V'$ a smooth variety birationally
  equivalent to $V$ (the choice does not matter).
  If $|\kpow{V'}{n}| = 0$ for all $n>0$, then
  we define $\kappa_V = -\infty$.  If $|\kpow{V'}{n}| \le 1$ for all $n > 0$
  with equality for at least one $n$, then define $\kappa_V = 0$.
  Otherwise, there exist positive integers $k,n$ and positive real numbers
  $c, C$ such that
  $cm^k < \dim |\kpow{V'}{mn}| < Cm^k$ for sufficiently large $m$, and we define
  $\kappa_V = k$.
  If $\kappa_V = \dim V$ then $V$ is {\em of general type}.
\end{defn}

%can omit
%% \begin{myremark} This definition depends on a choice of $V'$, but
%%   $\kappa_V$ does not depend on this choice.  If a variety in
%%   characteristic $p$ is not birational to any smooth variety, its Kodaira
%%   dimension is not defined.
%% \end{myremark}

%can omit one sentence
We now introduce two well-studied families of cubic fields.

\begin{defn}\label{def:krs-kn}
  \cite[\S 3.I (*), (3.2)]{tv-rcs} Let $1 \le r \le s-2$.  Define
  $K_{r,s}$ to be the field $\Q(\alpha_{r,s})$, where $\alpha_{r,s}$ satisfies
  $x(x-r)(x-s)-1 = 0$.  Let $n > 0$.  Define $K_n$ to be $\Q(\alpha_n)$,
  where $\alpha_n$ is a root of $x^3 + (n+1) x^2 + (n-2) x - 1$.
\end{defn}

These fields have easily constructed sets of units of maximal rank:
the $K_n$ are Galois and so $\alpha_n$ and
$\alpha_n^\sigma$ are units where $\sigma \in \Aut(K_n)$, while
$\alpha_{r,s}, \alpha_{r,s}-r \in \O^\times_{K_{r,s}}$.  This implies in some cases
that the singularities of the Hilbert modular varieties associated to them
are relatively tractable: see Remark \ref{rem:restrict-field}.

Kn\"oller proves the following \cite[Satz 1, Satz 2]{knoller}:
\begin{thm}\label{thm:general-type-knoller} \begin{enumerate}
  \item If $\O_{K_n}$ is generated by $\alpha_n$ (this holds if $n^2-n+7$
    is squarefree) and has class number $1$ and $n > 4$, then
    $H_{K_n}$ is of general type.
  \item If $\O_{K_{r,s}}$ is generated by $\alpha_{r,s}$ and has class number $1$
    and $3\zeta_{K_{r,s}}(2) D_{K_{r,s}} > 4\pi^6$, and $r > 1$,
    then $H_{K_{r,s}}$ is of general type.
  \end{enumerate}
\end{thm}

Every $H_K$ that Kn\"oller proves to be of general type satisfies
$p_g(H_K) > 1$.  Grundman proved the following
statement:
\begin{thm} \cite[Theorem 3]{gr-dcs}, \cite[Theorem 3]{gr-ds}
  Let $K \in \{K_{1,7}, K_{2,5}, K_{3,5}\}$.   Then
  $H_K$ is of geometric genus $0$ and positive Kodaira dimension.
\end{thm}
This was surprising because a Hilbert modular surface for level $1$ is
rational if and only if its geometric genus is $0$ \cite[Proposition VII.6.1]{vdg}.
\subsection{Summary of results}\label{subsec:summary}
We now state our main result.

\begin{copythm}{\ref{thm:general-type-from-473}, \ref{thm:general-type-nonprincipal-except-229}}
\begin{enumerate}
\item {\em Let $K$ be a cubic field with discriminant $\ge 473$ such the
  geometric genus of $H_K$ is $0$ or $1$.
  Then $H_{K}$ is of general type, unless possibly $K \cong \Q[x]/(x^3-7x-5)$
  is the cubic field with LMFDB label \cf{697} or
  $K \cong \Q[x]/(x^3-x^2-7x-3)$ is the cubic field with LMFDB label \cf{788}.}
\item {\em Let $K$ be a cubic field such that the
  geometric genus of $H_K$ is $0$ or $1$.  Let $A$ be an ideal
  that is not narrowly principal (this implies that $h_K^+ = 2$).
  Then $H_{K;A}$ is of general type, except
  possibly for the field \cf{229} defined by $x^3-4x-1$.}
\end{enumerate}
\end{copythm}

%% \smallskip\noindent {\bf Theorem \ref{thm:general-type-from-473}, \ref{thm:general-type-nonprincipal-except-229}.}
%% \begin{enumerate}
%% \item {\em Let $K$ be a cubic field with discriminant $\ge 473$ such the
%%   geometric genus of $H_K$ is $0$ or $1$.
%%   Then $H_{K}$ is of general type, unless possibly $K \cong \Q[x]/(x^3-7x-5)$
%%   is the cubic field  with LMFDB label \cf{697}.}
%% \item {\em Let $K$ be a cubic field such that the
%%   geometric genus of $H_K$ is $0$ or $1$.  Let $A$ be an ideal
%%   that is not narrowly principal (this implies that $h_K^+ = 2$).
%%   Then $H_{K;A}$ is of general type, except
%%   possibly for the field \cf{229} defined by $x^3-4x-1$.}
%% \end{enumerate}
%% \smallskip

We are able to prove something about four more fields:

\begin{copyprop}{\ref{prop:disc-469}, \ref{prop:kappa-ge-0-for-229}}
\begin{enumerate}
  \item {\em Let $K_1$ be the cubic field $\cf{404}=\Q(\alpha)$, where $\alpha$
    is a root of $x^3 - x^2 - 5x - 1$.  Then $\kappa_{H_{K_1}} \ge 0$.}
  \item {\em Let $K_2$ be the cubic field $\cf{469}=\Q(\alpha)$, where $\alpha$
    is a root of $x^3 - x^2 - 5x + 4$.  Then $\kappa_{H_{K_2}} > 0$.}
  \item {\em Let $K_3$ be the cubic field $\cf{788}=\Q(\alpha)$, where $\alpha$
    is a root of $x^3 - x^2 - 7x - 3$.  Then $\kappa_{H_{K_3}} > 0$.}
  \item {\em Let $K_4$ be the cubic field $\cf{229}=\Q(\alpha)$, where $\alpha$
    is a root of $x^3 - 4x - 1$, and let $A$ represent the nonprincipal genus.
    Then $\kappa_{H_{K_4;A}} \ge 0$.}
\end{enumerate}
\end{copyprop}

%% \smallskip\noindent {\bf Proposition \ref{prop:disc-469}, \ref{prop:kappa-ge-0-for-229}}
%% \begin{enumerate}
%%   \item {\em Let $K_1$ be the cubic field $\cf{404}=\Q(\alpha)$, where $\alpha$
%%     is a root of $x^3 - x^2 - 5x - 1$.  Then $\kappa_{H_K} \ge 0$.}
%%   \item {\em Let $K_2$ be the cubic field $\cf{469}=\Q(\alpha)$, where $\alpha$
%%     is a root of $x^3 - x^2 - 5x + 4$.  Then $\kappa_{H_K} > 0$.}
%%   \item {\em Let $K_3$ be the cubic field $\cf{229}=\Q(\alpha)$, where $\alpha$
%%     is a root of $x^3 - 4x - 1$, and let $A$ represent the nonprincipal genus.
%%     Then $\kappa_{H_{K;A}} \ge 0$.}

%% \end{enumerate}
%% % also 257 nonprincipal genus
%% \smallskip
  
\begin{myremark}\label{rem:geometric-genus}
  It would seem more natural to
  study the $H_{K;A}$ with $p_g \le 1$ rather than the
  $H_{K;A}$ for which $p_g(H_K) \le 1$.  In Appendix~\ref{app:same-dimension}
  we will show that when $K$ is a cubic field the dimensions of the spaces
  of modular forms for all of the $\Gamma_0(I;A)$ are equal.  Since the
  number of cusps is the same as well, the geometric genera are equal, as
  we will see in Remark~\ref{rem:defects-small}.
  So these conditions are equivalent.
  %% For a single field there is
  %% no difficulty in doing so, but we do not know how to obtain
  %% a reasonable bound on the discriminant of a field $K$ for which there
  %% exists a component of $H_K$ with $p_g \le 1$.
  %%% see John's email
\end{myremark}

\begin{myremark}\label{rem:why-probably-enough}
  Our method determines two
  constants $c_1(K) = -2\zeta_K(-1), c_2(K)$ associated to a field such that if
  $c_1(K) > c_2(K)$ then $H_K$ is of general type.  Both of these grow
  without extreme oscillations
  (see Tables \ref{tab:cubic-pg-0},~\ref{tab:cubic-pg-1}), and for
  the cubic field \cf{1937} of maximal discriminant for
  which $p_g(H_K) \le 1$ we already have $c_1(K) = 28, c_2(K) = 397/30$.
  Therefore we believe that if $p_g(H_K) > 1$ (which already implies that
  $\kappa_{H_K} > 0$) then all $H_{K;A}$ are of general type, and that this can
  be verified for all such $K$ by the methods of this paper.
  In order to prove this it would be necessary to verify that $c_1 > c_2$
  for all cubic fields violating
  the inequality $\frac{D_K \zeta_K(2)}{hR} \ge 2^{3-2}\pi^3$ of
  \cite[Corollary 10]{gr-cl} and all choices of $A$.
  The amount of computer time required is substantial but not prohibitive,
  especially since it appears that this inequality fails for only $421$
  cubic fields.  See Section~\ref{sec:pg-greater-than-1} for more details.
\end{myremark}

Previous work on the Kodaira dimension of Hilbert modular varieties
of dimension greater than $2$ has
largely been restricted to level $1$, perhaps because it is easier to
compute invariants under this assumption and because
$\kappa_{H_{K,I;A}} \ge \kappa_{H_{K;A}}$ (in characteristic $0$
this is a general statement about dominant maps).
In Propositions \ref{cor:pg-01-prime-table},~\ref{cor:pg-01-table}
we classify the $H_{K,I}$ for which
the geometric genus $p_g$ is at most $1$.
We would like to prove results analogous to those of \cite{hamahata},
studying the Kodaira dimensions of the $H_{K,I}$ of geometric genus at most $1$.
We can show that many of the $H_{K,I}$ are of general
type even when we cannot prove this for $H_K$.
Here is a sample of our results.  We use $\p_q$ to refer to a
prime of norm $q$; it does not matter which one we
choose because of the Galois automorphism.  The full result is presented in
Propositions~\ref{cor:pg-01-prime-table} to \ref{cor:mostly-gt-not-prime}.
See also Tables \ref{table:pg-01-prime},~\ref{table:pg-01-ngt}.

\begin{copythm}{\ref{thm:pg-01-disc-49}}
Let $K = \qz7$.  The Hilbert modular variety $H_{K,I}$ is
  of geometric genus $0$ if and only if
  $I \in \{(1), \p_7, (2), \p_{13}, \p_{29}, \p_{43}\}$, and $1$ if and only if
  $I \in \{(3), \p_{41}, \p_7^2, 2\p_7, (4), \p_{71}, \p_{7}\p_{13}, \p_{97}, \p_{113}, \p_{127}, \p_{13}^2\}$.
\end{copythm}

%% \smallskip\noindent {\bf Theorem \ref{thm:pg-01-disc-49}}
%% {\em Let $K = \qz7$.  The Hilbert modular variety $H_{K,I}$ is
%%   of geometric genus $0$ if and only if
%%   $I \in \{(1), \p_7, (2), \p_{13}, \p_{29}, \p_{43}\}$, and $1$ if and only if
%%   $I \in \{(3), \p_{41}, \p_7^2, 2\p_7, (4), \p_{71}, \p_{7}\p_{13}, \p_{97}, \p_{113}, \p_{127}, \p_{13}^2\}$.}
%% \smallskip

\begin{copyprop}{\ref{prop:mostly-gt}, \ref{prop:pos-kod-dim}}
All of these varieties are of general type, with the possible
  exceptions of $I \in \{(1),\p_7,(2)\}$.  We have $\kappa_{H_{K,(2)}} > 0$.
\end{copyprop}

%% \noindent {\bf Proposition \ref{prop:mostly-gt}, Proposition \ref{prop:pos-kod-dim}}
%% {\em All of these varieties are of general type, with the possible
%%   exceptions of $I \in \{(1),\p_7,(2)\}$.  We have $\kappa_{H_{K,(2)}} > 0$.}
%% \smallskip

\begin{myremark} We would like to study the varieties not shown to be of
  general type and understand their geometry explicitly as in \cite{logan}.
  However, this seems to be a very difficult problem, and
  we have no new results in this direction.
\end{myremark}
%% \begin{prop}\label{prop:pg-148} (cf.~Corollaries \ref{cor:pg-01-prime-table}, \ref{cor:pg-01-table}) %%% add reference to Kodaira dimensions
%%   Let $K$ be the real cubic field of discriminant $148$.  Then 
%%   $p_g(H_{K,I}) = 0$ if and only if
%%   $I \in \{(1),\p_2,\p_2^2, \p_5, \p_2 \p_5, \p_{13}\}$, and
%%   $p_g = 1$ if and only if $I \in \{(2),\p_{17}, \p_2^2 \p_5, \p_{25}\}$.
%%   In all cases except possibly the powers of $\p_2$ we have
%%   $\kappa_{H_{K,I}} > 0$. %%% recheck p_2^2 and p_2^3
%% \end{prop}

% first few paragraphs can go
%\subsection{Proving that moduli spaces are of general type}\label{subsec:general-type}
\subsection{Method for studying the Kodaira dimension of Hilbert modular threefolds}\label{subsec:general-type}

In order to prove that Hilbert modular threefolds are of general type, we
will show that powers of the canonical bundle have many sections.  Here
we describe these sections as Hilbert modular forms; later, in
Section~\ref{sec:defects}, we explain which ones belong to $H^0(nK)$, and
in Sections~\ref{sec:algorithms}--\ref{sec:defect-algs} we
will give methods for exact and asymptotic calculation of the dimensions.
In Sections \ref{sec:results-level-1},~\ref{sec:results-higher-level} we present
the results.  The statements of this section and Section~\ref{sec:defects} are
not new, going back to \cite{knoller}.  However, results like those stated in
Section~\ref{subsec:summary} and proved in 
subsequent sections have previously appeared only in special cases.

\begin{notation}\label{not:embeddings}
  Fix an ordering of the $d$ real embeddings $\pi_i: K \hookrightarrow \R$.
  For $x \in K$ and $1 \le i \le d$, let $x_i = \pi_i(x)$.
  For $z \in \HH^d$, let $z_i$ be the $i$th component.
\end{notation}

\begin{defn} \cite[Definition~I.6.1]{vdg}
  Let $k$ be a nonnegative even integer.
  A {\em Hilbert modular form} of (parallel) weight $2k$
  for a group $\Gamma$ on an open subset $U \subseteq \HH^d$
  is a holomorphic function satisfying
  \begin{equation*}
    f(\gamma z) = \left(\prod_{i=1}^d (c_i z_i + d_i)^{2k} (\det \gamma)_i^{-k}\right) f(z)
  \end{equation*}
  for all pairs $(\gamma,z)$ with $\gamma \in \Gamma$ and
  $z, \gamma z \in U$.
  We denote the space of Hilbert modular forms of weight $2k$ by
  $M_{2k}(U;\Gamma)$, omitting $U$ or $\Gamma$ when these are clear or
  unimportant.
\end{defn}

\begin{myremark} The books \cite{freitag,vdg} are invaluable general references
  for Hilbert modular forms and (for the second) Hilbert modular varieties.
  Magma \cite{magma} offers built-in functions for computing dimensions of
  spaces of Hilbert modular forms and coefficients of their Fourier expansions,
  though it is limited to certain levels for fields of odd degree.
  The article \cite{hmf-mult} gives a more advanced introduction to some
  computational issues for Hilbert modular forms and varieties,
  especially in degree $2$, and to the use of the more recent software
  \cite{hmf} that performs a wide range of essential computations for
  Hilbert modular forms over quadratic fields and Hilbert modular surfaces.
\end{myremark}

\begin{myremark}\label{rem:hmf-weight-2k-can-k}
  An easy calculation \cite[Lemma III.3.1]{vdg} shows that if $f$ is a Hilbert
  modular form of weight $2k$ then $f(z_1,\dots,z_n) (dz_1\,\dots\,dz_n)^{\otimes k}$
  is invariant under $\Gamma$.
  In particular, if $\Gamma \backslash U$ is nonsingular then
  $\HHH^0(\Gamma \backslash U,\kpow{U}{k})$ is naturally identified with 
  $M_{2k}(U;\Gamma)$.  However, this does not imply that
  the space $\HHH^0(H_K, \kpow{U}{k})$ of modular forms of weight $k$
  can be identified with $M_{2k}(\Gamma)$, because there are additional
  conditions at the cusps and elliptic points for a modular form to give
  a section of $\kpow{U}{k}$.  Rather, as in the classical case,
  the space of modular forms of weight $2k$ is identified with
  $\HHH^0(H_K,(K(\textup{log cusps}))^{\otimes k})$, and
  Proj of the ring of modular forms is the {\em Baily-Borel compactification}
  of $H_K$ \cite[Theorem II.7.1]{vdg}.  In particular $H_K$ is always of log
  general type.  
\end{myremark}

We recall some facts about the rate of growth of the dimension of $M_{2k}$.
Although an exact formula for the dimension requires a detailed examination
of the elliptic points (Definition~\ref{def:elliptic-pt}), the asymptotic
formula does not require this.  The starting point is a result of Siegel:

\begin{thm}\cite[Theorem IV.1.1]{vdg}\label{thm:fd-volume}
  Relative to the standard hyperbolic metric on $\HH^d$, the volume of
  $\PSL_2(\O_K)\backslash \HH^d$ is equal to $2\zeta_k(-1)$.
\end{thm}

Specializing \cite[Theorem 11]{shimizu}, Thomas and Vasquez give a formula
for the dimension of the space of modular forms for $\PSL_2(\O_K)$
and its torsion-free finite-index subgroups.
The following well-known result follows from this:

\begin{prop}\label{prop:asymptotic-dim-hmf}
  Let $\Gamma$ be a group commensurable with $\PSL_2(\O_K)$, where $K$ is
  a totally real field of degree $d$.  Then the
  dimension of the space of cusp forms of weight $2k$ is asymptotically
  equal to $2\zeta_K(-1)(-k)^d [\PSL_2(\O_K):\Gamma]$.
\end{prop}

\begin{proof} In the case where $\Gamma = \PSL_2(\O_K)$, this follows from
  \cite[(2.2)]{tv-rhmf}, since in their terminology $\PSL_2(\O_K)$ is always
  of modular type.  To pass to a commensurable subgroup we use
  \cite[Theorem II.3.5]{freitag} to see that the leading coefficient
  is proportional to $[\PSL_2(\O_K):\Gamma]$, it being clear that the
  terms in the formula other than $\vol(\Gamma\backslash \HH^n)(2r-1)^n$
  are of lower order.
\end{proof}

%% \begin{myremark}\label{rem:always-parallel}
%%   In this paper we will only use modular forms for which all the
%%   $k_i$ are equal, and thus we will write the weight as a single integer
%%   rather than a sequence.
%% \end{myremark}

We now consider the special points of Hilbert modular varieties, namely
fixed points of nonidentity elements of $\Gamma$ and the cusps used
to compactify $\Gamma \backslash \HH^d$.  The first of these must be studied
in order to obtain an exact formula for the dimension of the space of modular
forms, while the second are vital for estimating the difference between this
dimension and the dimension of $\HHH^0$ of powers of the canonical line bundle.

\begin{defn}\cite[p. 15]{vdg}\label{def:elliptic-pt} An {\em elliptic point} of
  $H_{K,I;A}$ is the image of the fixed point of an element of
  $\Gamma_0(I;A)$ of finite order greater than $1$.
\end{defn}

\begin{notation}\label{not:cusp-ells}
  Let $C_{K,I;A}, E_{K,I;A}$ be the sets of cusps and elliptic
  points of $H_{K,I;A}$ respectively.
\end{notation}

\begin{defn}\cite[Definition I.6.2]{vdg}
  Let $U$ be an open subset of $\HH^d \setminus E_{K,I;A}$.
  Let $C_U$ be the subset of $C_{K,I;A}$ consisting of cusps
  such that $U$ contains a punctured neighbourhood of some representative of
  $C$, and similarly for $E_U$ and elliptic points.
  A {\em Hilbert cusp form} on $U$
  is a Hilbert modular form on $U$ that extends to
  a holomorphic function on $U \cup \Gamma \backslash (C_U \cup E_U)$
  that is $0$ on all cusps of $C_U$.
\end{defn}

\begin{thm}\label{thm:can-div-hmv}\cite[Section 3.6, Lemma]{hirz}
  For a resolution of singularities ${\tilde H_{K,I;A}}$
  of $H_{K,I;A}$, the sections of $K_{\tilde H_{K,I;A}}$ are precisely the Hilbert
  cusp forms of weight $2$.
\end{thm}

Hirzebruch only states this for quadratic fields, but the proof does not
use this assumption.
This theorem does not extend to powers of $K_{\tilde H_{K,I;A}}$;
the sections of $nK$ are not simply the cusp forms of weight $2n$.
To describe the situation, we introduce some further notation.

\begin{notation}\label{not:neighbourhoods}
  Let $P$ be a cusp or elliptic point for $H$.  Let $U_P$ be a small
  punctured neighbourhood of $P$ (a complex manifold) and let
  $V_P = U_P \cup \{P\}$.  Let $\tilde V_P$ be any resolution of the singular
  point $P$ of $V_P$ (nothing depends on the choice).
  The natural injection
  $\HHH^0(\tilde V_P,\kpow{\tilde V_P}{q}) \hookrightarrow \HHH^0(U_P,\kpow{U_P}{q})$
  will be denoted $e_{P;q}$.
\end{notation}

\begin{defn}\cite[p.~6]{knoller}
  The {\em $q\/$th defect} $\delta_P(q)$ of $P$ is
  $\dim \coker e_{P;q}$.
\end{defn}

\begin{myremark}\label{rem:defects-small}
  %For $q = 0$ the map $e_{P;q}$ is the identity map on the $1$-dimensional
  %space of constant functions, so $\delta_P(0) = 0$ for all $P$.
  We have $\delta_P(0) = 0$ for all $P$; for elliptic points this is
  the case $m=0$ of \cite[Satz 2.0]{knoller}, while for cusps it is a direct
  consequence of \cite[Satz 2.4]{knoller} (see
  Definition~\ref{def:knoller-notation} for an explanation of the notation).
  For $q = 1$, Theorem~\ref{thm:can-div-hmv} shows that $\delta_P(1) = 1$
  if $P$ is a cusp, at least if $[K:\Q] > 1$.  In addition, if $P$ is
  an elliptic point then $\delta_P(1) = 0$.
  This was first shown by Freitag; it is clearly explained 
  on \cite[p. 56]{vdg}.
\end{myremark}

Now let $S = C_{K,I;A} \cup E_{K,I;A}$ and let $U = H_{K,I;A} \setminus S$,
which is a smooth complex manifold.
For each $P \in S$, let $\tilde V_P$ be a resolution of $U \cup P$, and let
$\tilde H_{K,I;A}$ be a simultaneous resolution of all the singularities,
obtained by identifying the $\tilde V_P$ along the open subsets identified
with $U$.  We then have
$\HHH^0(\tilde H_{K,I;A},K^{\otimes q}) = \bigcap_{P \in S} \im e_{P;q}$.  It follows that
$\hhh^0(\tilde H_{K,I;A},K^{\otimes q}) \ge \dim M_{2q} - \sum_{P \in S} \delta_P(q)$,
since $M_{2q} = \hhh^0(H_{K,I;A} \setminus S,K^{\otimes q})$ (in contrast to the
case $K = \Q$, where we need an additional condition to ensure that modular
forms are holomorphic at the cusps; see \cite[Section I.6]{vdg}).
We thus obtain a lower bound for $\hhh^0(K^{\otimes q})$,
namely
\begin{equation}\label{eqn:lower-bound}
  \hhh^0(K^{\otimes q}) \ge \dim M_{2q} - \sum_{P \in S} \delta_{P}(q).
\end{equation}
If we can prove that this lower bound is positive (resp.~greater than $1$)
for some $q$, then we have shown that the Kodaira dimension of the Hilbert
modular variety is nonnegative (resp.~positive).  If we can estimate the
lower bound as being at least $cq^d$ for some $c>0$, the variety is of
general type.  
Grundman did this and essentially proved in \cite[Theorem 3]{gr-dcs},
\cite[Theorem 3]{gr-ds} that the component of the Hilbert modular
threefolds corresponding to the principal genus
for \cf{697} is of positive Kodaira dimension,
while for \cf{761}, \cf{985} it is of general type.

\begin{myremark}\label{rem:grundman-general-type}
  %% These fields are of the form $K_{a,b}$ (Definition~\ref{def:krs-kn})
  %% for some small $a,b$.  In each
  %% case $\alpha$ is a totally positive unit that is not a square and the
  %% narrow class number is $2$.  Thus there are two Hilbert modular
  %% varieties of level $1$ (cf.~Remark \ref{rem:hmv-components}).
  Grundman states only that the plurigenera are not equal to $0$ for the
  principal genus, concluding that the varieties are not rational.
  However, the statement that $H_K$ is of general type for the second and third
  fields follows from her results.
  We will explain this in Example~\ref{ex:gr-761} and
  Remark~\ref{rem:gr-985}.
\end{myremark}

\subsection*{Acknowledgments} This paper is inspired by the work done
on Hilbert modular forms in the context of the Simons Collaboration on
Algebraic Geometry, Number Theory, and Computation.  Some of
the code from \cite{hmf} was used in this work.  I would like
to thank all of the authors of \cite{hmf,hmf-mult} and especially
Eran Assaf, Abhijit Mudigonda, and John Voight for their patient replies to my
numerous questions, Colin Ingalls for a helpful conversation on
quotient singularities, and Leland McInnes for advice on producing interactive
three-dimensional figures.  I would also like to thank the referees for their
detailed and valuable criticism which led to substantial improvements in the
exposition and the technical reviewer for pointing out some bugs and flaws
in the code.

In addition,
I thank the Department of Pure Mathematics at the University of Waterloo,
where this work was begun, for its hospitality, and
the Tutte Institute for Mathematics and Computation for its support of
my external research and for the opportunity to take leave from my
employment there.

\section{Defects}\label{sec:defects}
Thomas and Vasquez give a formula \cite[Theorem 3.10]{tv-rhmf}
for the dimension of $M_{2q}$ for
certain special subgroups of $\PSL_2(K)$.
A general formula for $\Gamma_0(I)$ does not seem to appear in one
single place in the literature, but can be extracted from
\cite[39.10]{voight} and \cite[5.1]{hmf-mult}.
The calculation is implemented in a Magma script \cite{code}
(I thank John Voight for many clarifying comments).
This leaves the problem of computing the defects.

Thomas \cite[Section 1]{thomas-dcs},
referring to \cite{knoller}, gives a method for evaluating
the defect in terms of a count of
elements with multiples of bounded trace.  We give the details in
Definition~\ref{def:knoller-notation}.

\begin{defn}\label{def:knoller-notation}\cite{knoller}
  Let $M \subset K$ be a free abelian subgroup of rank $d = [K:\Q]$ and
  $V \subseteq \otp$ a group of totally positive units
  preserving $M$.
  Following Kn\"oller, we consider the cusp singularity at the image of
  $\infty$ in the quotient of $\HH^d \cup \{\infty\}$ by the
  group of matrices
  $\left\{\begin{psmallmatrix}v&g\\0&1\end{psmallmatrix}: v \in V, g \in M\right\}$:
  we refer to this as a cusp of {\em type $(M,V)$}.  Let $\delta_{(M,V)}(q)$
  be the $q$th defect of a cusp of type $(M,V)$.
  Let $\hat M = \{x \in K: \tr xm \in \Z \textup{ for all } m \in M\}$
  be the dual of $M$ with respect to the trace.  For $q \in \N$, let
  $\Lambda_q(M) = \{x \in {\hat M}_+: \tr xm < q \textup{ for some } m \in M_+\} \cup \{0\}$.
\end{defn}

For any such $M$, a suitable subgroup $V$ exists:
\begin{lem}\label{lem:group-preserved}
  Let $M \subset K$ be a free abelian subgroup of rank $d$.  The subgroup
  of $\otp$ preserving $M$ is of finite index in $\otp$.
\end{lem}
  
\begin{proof} By multiplying by a suitable positive integer we may assume
  that $M \subset \O_K$.  Then $\otp$ acts on subgroups
  $N \subseteq \O_K$ with $\O_K/M \cong \O_K/N$.  There are finitely
  many such subgroups, so the stabilizer of any one of them is of finite index.
\end{proof}

In \cite[Satz 2.4]{knoller} Kn\"oller proves:

\begin{thm}\label{thm:knoller-defect} 
  The $q$th defect of a cusp of type $(M,V)$ is equal to
  $\#(\Lambda_q(M)/V)$.
\end{thm}

\begin{myremark} In particular, we always have $\Lambda_1 = \{0\}$ and so
  $\delta_{M,V}(1) = 1$.  This is closely related to
  Theorem~\ref{thm:can-div-hmv}.
\end{myremark}

\begin{myremark}\label{rem:restrict-field}
  Thomas states that it is difficult to calculate $\delta_{M,V}(q)$ in general
  and imposes an additional restriction
  that implies that every element of $\Lambda_q(\O_K)/V$ is represented by
  $x \in \hat \O_{K,+}$ with $\tr x < q$.  Later works of
  Thomas-Vasquez \cite{tv-rcs} and Grundman \cite{gr-dcs,gr-ds}
  follow Thomas in imposing this restriction.
  This holds for the $K_{r,s}$ and $K_n$
  (Definition~\ref{def:krs-kn}) under the additional hypothesis that a root
  of the given defining polynomial generates the ring of integers
  \cite[Section 3]{tv-rcs}.
  In these cases they construct explicit cusp resolutions: see
  \cite[Section 3]{tv-rcs}, \cite[Section 6]{gr-dcs}.
  %% Note that in both cases it is easy to describe a set of independent units
  %% of maximal rank.  The polynomial defining $K_n$ has square discriminant and
  %% thus defines a Galois extension of $\Q$, and clearly a root $\alpha$
  %% of the polynomial is a unit, so $\alpha, \alpha^\sigma$ are independent units
  %% where $\sigma$ generates the Galois group.  For the $K_{r,s}$,
  %% if $\beta(\beta-r)(\beta-s) = 1$
  %% then visibly $\beta, \beta-r$ are independent units.
  
  In this work we will describe practical methods for computing $\delta_{M,V}$
  without an explicit cusp resolution.
  %We do not describe the cusp resolutions
  %explicitly, although its properties are used implicitly.
  This allows us to prove results analogous to those of \cite{tv-rcs,gr-ds} for
  general totally real cubic fields.
\end{myremark}

\begin{question}\label{q:only-these-cubics}
  Let $K$ be a real cubic field such that for all $q$, 
  every element of $\Lambda_q(\O_K)/V$ is represented by an element
  with trace less than $q$.  Must $K$ belong to one of the families
  $K_{r,s}$ or $K_n$?  Is there a similar classification for fields of
  higher degree?
\end{question}

\begin{myremark}\label{rem:which-larger-fields}
  The definition of the fields $K_{r,s}$ can be extended to
  arbitrary degree, defining $K_{r_1,\dots,r_n}$ to be the field of degree $n+1$
  defined by the polynomial $x \prod_{i=1}^n (x-r_i) - 1$.  For $n = 1$,
  we obtain the quadratic fields $\Q(\sqrt{n^2+4})$.  However,
  these do not necessarily have the property of
  Question~\ref{q:only-these-cubics}.  For example, with $n = 16$ this field is
  $\Q(\sqrt{65})$, in which
  $\frac{9+\sqrt{65}}{2} \cdot \frac{9-\sqrt{65}}{2} = 4$.  One
  easily concludes that $\frac{9+\sqrt{65}}{2}$ is an element of
  $\Lambda_9(\O_K)$ having no multiple by a totally positive unit of trace
  less than $9$.
  %% Likewise, in degree $4$, the field $K_{2,4,5}$
  %% has class number $1$,
  %% ring of integers generated by $\alpha$, and unit group generated by
  %% $-1$ and $\alpha - i$ for $i \in \{0,2,4,5\}$.  Nevertheless, setting
  %% $$\rho = -\alpha^3+7\alpha^2-9\alpha, \quad \sigma = 2876267-2455291\alpha+657471\alpha^2-54847\alpha^3,$$
  %% we find that $\tr \rho\sigma = 2231968 < \tr \sigma = 3278203$, even though
  %% $\sigma$ cannot be multiplied by a totally positive unit to obtain an
  %% element of smaller trace.
  In Section~\ref{sec:trace-min-reducer} we will
  present systematic methods to find such elements or prove that they
  do not exist.
\end{myremark}
  
We now return to describe the defects of cyclic quotient singularities,
following \cite[II.6, III.3]{vdg}.  Van der Geer only treats the case of
dimension $2$, but his methods apply without change to the general situation.
Thus we only state the results.

\begin{notation} Let $a_1, \dots, a_n$ be positive integers relatively
  prime to $m$ and less than $m$.  Let $L$ be the sublattice of $\Z^n$
  consisting of vectors $(x_1,\dots,x_n)$ with
  $m|\sum_{i=1}^n a_i x_i$.  The dual lattice, which is
  generated over $\Z^n$ by $(a_1,\dots,a_n)/m$, is denoted $M$.
  For $k = 1, \dots, m-1$ let $P_k = \frac{1}{m} (k a_j \bmod m)_{j=1}^n \in M$,
  and define a simplex $T_k$ by the conditions $x_1, \dots, x_n \ge 1$ and
  $\sum_i {P_k}_i x_i < 1$.  (Note that $T_k = \emptyset$ if
  $\sum_i {P_k}_i \ge 1$.)  Let $T = \cup_i T_i$ and let $T(q)$ be $T$
  scaled up by $q$.
\end{notation}

\begin{defn}\label{def:cyclic-quotient} 
  Let $m \in \Z^+$ and let $a_1, \dots, a_n$ be positive integers less than $m$
  and relatively prime to $m$.
  Let the cyclic group $C_m$ of order $m$ act on $\A^n$ such that a generator
  acts by $(x_1,\dots,x_n) \to (\mu_m^{a_1} x_1,\dots,\mu_m^{a_n} x_n)$.
  A singularity locally isomorphic to that 
  of $\A^n/C_m$ at the origin is called a
  {\em cyclic quotient singularity of type} $(a_1,\dots,a_n;m)$.
\end{defn}

Such a singularity is isolated by \cite[Corollary 2.2]{morrison-stevens}.

\begin{thm}\label{thm:ell-defect}\cite[p. 56]{vdg}
  The $q$th defect of a cyclic quotient singularity of type
  $(a_1,\dots,a_n;m)$ is the number of $L$-points of $T(q)$.
\end{thm}

\begin{cor}\label{cor:ell-defect-asymptotic}
  Asymptotically the $q$th defect is equal to
  $\frac{\vol T}{m} q^d + O(q^{d-1})$. \qed
\end{cor}

\begin{ex}\label{ex:cq-7-9}
  We consider the elliptic points of order $7, 9$ that arise for
  $\qz7$ and $\qz9$.  For the former, we have one point of
  each of the types $(1,\pm 3, \pm 2;7)$ by \cite[Proposition 2.10 (ii)]{tv-cn}.
  The $(1,4,2;7)$ point has all defects $0$,
  while the other three
  are isomorphic.  Indeed, given a $(1,3,2;7)$-singularity, we may replace
  the generator $g$ of $C_7$ by $g^4$ to see that it is also a
  $(4,5,1;7)$-singularity, or by $g^5$ to see that it is a singularity of type
  $(5,1,3;7)$.  The order of the $a_i$ is of no importance.
  We quickly compute that for each of them there is a
  single simplex of volume $1/36$, giving an asymptotic of $q^3/252$
  for the $q$th defect.  The first $10$ defects are $0,0,0,0,0,0,1,1,2,3$.

  Similarly, for $\qz9$, according to
  \cite[Proposition 2.10 (iii)]{tv-cn} the types are $(1,\pm 2, \pm 4;9)$,
  each occurring once.  All defects are $0$ for the point of type $(1,7,4;9)$
  and again the other three are equivalent.
  For $(1,2,4;9)$ and those isomorphic to it there are two simplices of
  volume $1/6, 1/60$, but the second is contained in the first so the volume
  of the union is just $1/6$.  So the asymptotic is $q^3/54$; the initial
  defects come to $0,0,0,0,1,2,4,6,10,14$.

  We display the relevant simplices in Figure \ref{fig:quotient-simplices}.
  \begin{figure}
    \includegraphics[scale=0.7]{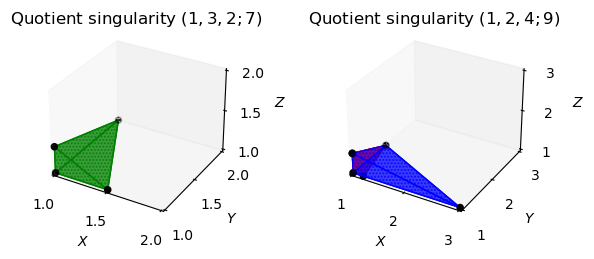}
    \caption{Simplices that describe the defects of $(1,3,2;7)$ and
      $(1,2,4;9)$-singularities.  All simplices are defined by the condition
      that all coordinates are at least $1$ and by one additional inequality.
      On the left, the choice $k = 1$ gives a single simplex defined
      by the inequality $x + 3y + 2z \le 7$.
      On the right, the blue simplex is defined by $x + 2y + 4z \le 9$ and the
      red simplex (contained in it) by $5x + y + 2z \le 9$.  Note that the
      scales are different.}\label{fig:quotient-simplices}
  \end{figure}
\end{ex}
  
\begin{myremark}\label{rem:cqs-2-3-canonical}
  As pointed out in \cite[Satz 2.0]{knoller}, the defects of
  cyclic quotient singularities of order $2, 3$ on Hilbert modular threefolds
  are $0$; this is immediate from Theorem~\ref{thm:ell-defect} as well.
  Thus, in light
  of the classification of cyclic quotient singularities of Hilbert modular
  varieties of cubic fields \cite[Theorem 2.12]{tv-rhmf} (less accessibly,
  though earlier, in \cite{gundlach}), the cyclic quotient singularities
  make no contribution to the defect for fields 
  other than $\qz7, \qz9$.

  I thank the referee for pointing out that
  this is analogous to the situation for quadratic
  fields, for which the fields of smallest discriminant
  $\Q(\sqrt 5), \Q(\sqrt 2), \Q(\sqrt 3)$ admit special quotient singularities
  that need to be considered separately (these are described in
  \cite[Table 1.III]{vdg}).
  In general \cite[(7)]{tsuyumine} gives all totally
  real fields for which there are elliptic points with nonvanishing defects.
  Excluding quadratic fields and quartic fields containing $\sqrt{3}$, the list
  is finite.
\end{myremark}
  
In the next two sections of this paper we describe methods for
calculating the sets $\Lambda_q(M)/V$ and estimating their asymptotic
growth that can be applied to general totally real
fields.  The most interesting case for us is that of
$M = \O_K, V = \otp$, which is associated to the single cusp of
$H_K$ when $K$ has class number $1$, but imposing this restriction does not
simplify the method, and we will consider more general singularities in
Sections~\ref{sec:results-level-1},\ref{sec:results-higher-level}.
%% Previous work on the
%% Hilbert modular varieties associated to fields in the families $K_n, K_{r,s}$
%% was strongly dependent on the existence of a special resolution of the cusp
%% and the
%% consequence that every element of $\Lambda_q(\O_K)$ is represented by an
%% element of trace less than $q$.  Our method does not require the explicit
%% construction of a resolution.

\section{Basic algorithms}\label{sec:algorithms}
Here we describe the algorithms that underlie our computations.  We assume
that the standard invariants of algebraic number theory can be computed;
in other words, given a number field $K$, we assume that its ring of integers,
different, class group, and unit group can be calculated, and that the prime
factorization of a fractional ideal can be determined.  Such calculations
can be performed by various symbolic computation systems, such as Magma
(used in writing the paper and needed to run the accompanying
scripts) and PARI/GP \cite{gp}.  In addition to these,
we need algorithms to find the elements of an ideal $I$ of $\O_K$
whose real embeddings
satisfy certain inequalities and to find the volume of a rational polytope.

\subsection{Rational points in a polytope}
\begin{defn} A {\em rational polytope} $P$ in $\R^n$ is an intersection of
  half-spaces defined by linear inequalities $\sum_{i=1}^n a_i x_i \le c$
  with $a_i, c \in \Q$
  whose $n$-dimensional Lebesgue measure is positive and finite.
\end{defn}

%% (The last condition could be stated in various ways.  For example, positivity
%% of the volume is equivalent to the set containing a nonempty open subset of
%% $\R^n$, and under this assumption
%% finiteness is equivalent to the boundedness of $P$.)

\begin{myremark} We will assume that the volume of a rational polytope can be
  computed given the facets.  For fixed dimension, this can be done in
  polynomial time \cite[p.~4033]{Enge2009}; in this paper we are always in
  $\R^3$.  By 
  \cite[Theorem 2]{tsuyumine}, as extended in
  Theorem~\ref{thm:tsuyumine-all-components},
  all components of Hilbert modular varieties are of general type for
  fields of degree greater than $6$, so in the application to Hilbert modular
  varieties we never need to go beyond $\R^6$ and the volume can be computed
  in polynomial time.
\end{myremark}

In order to find the integral points in a general (not necessarily rational)
polytope, we begin by finding a rational polytope that contains it.  This
can be done from crude bounds on the individual coordinates of every point
of the polytope.  More sophisticated algorithms exist and have been
implemented in software \cite{latte}, but we use a
simple method to find the integral points in a rational polytope:
\begin{alg}\label{alg:int-points-rat-poly}
  Given a rational polytope $P$, determine its integral points.
  \begin{enumerate}
  \item If $\dim P = 1$ the problem is trivial.
  \item For each coordinate $x_i$, determine lower and upper bounds
    $m_i, M_i$ for $x_i$.  If for any $i$ there is only one integer in
    $[m_i,M_i]$, then set $x_i$ to that integer and reduce the dimension by $1$.
  \item Choose the $i$ that maximizes $M_i-m_i$ and
    $c_i \in (m_i,M_i) \setminus \Z$ (in practice one chooses $c_i$ near
    $(m_i+M_i)/2$).  Apply the algorithm to the two polytopes
    $P \cap (x_i \ge \lceil c_i \rceil), P \cap (x_i \ge \lfloor c_i \rfloor)$.
  \end{enumerate}
\end{alg}

\begin{myremark} We have stated this for integral points, but the
  same algorithm can be used to determine the $P$-points of any rational
  lattice, since one can choose a rational change of coordinates taking such
  a lattice to $\Z^n$, and any such change takes a rational polytope to
  a rational polytope.
\end{myremark}

\begin{myremark}\label{rem:how-to-enumerate} Magma provides a built-in function
  to enumerate the integral points in a polytope; however, it is often slow.
  Empirically we have found that combining
  Algorithm~\ref{alg:int-points-rat-poly} with calls to
  the built-in function when the polytope is small enough gives better
  performance than either one does individually.
\end{myremark}

\subsection{Union of polytopes}
A second important problem is to compute the volume of a union of rational
polytopes $\cup_{i=1}^k P_i$.
We use the following algorithm.
\begin{alg}\label{alg:volume}
  Given a finite set of polytopes $P_1, \dots, P_k$, find
  $v = \vol(\cup_{i=1}^k P_i)$.
  \begin{enumerate}
  \item\label{item:initialize}
    Sort the $P_i$ in order of increasing volume to obtain a list
    $\PP_0 = \PP$.
    For each pair $(i,j)$
    with $i < j$, determine whether $P_i \subseteq P_j$, and if so remove
    $P_i$ from $\PP$.
    %(This is faster if one searches in decreasing order
    %of $j$ for each fixed $i$.)
    For each pair $(i,j)$ with $i < j$, determine whether
    $\vol(P_i \cap P_j) > 0$, and if so add it to the list $\LL$ of pairs.
  \item\label{item:while} While $\LL$ is nonempty, repeat the following steps:
    \begin{enumerate}
  \item\label{subitem:first} Choose an element $(i,j)$.
  \item Choose a face $F$ of $P_i$ that contains points of the interior of
    $P_j$ (if there is no such face then $(i,j) \notin \LL$).  Let $P_+, P_-$
    be the intersections of $P_j$ with the half-spaces defined by $F$.
  \item Remove $P_j$ from $\PP$ and all pairs containing $j$ from $\LL$.
  \item\label{subitem:contained}
    Determine whether $P_+, P_-$ are contained in any element of $\PP$.
  \item\label{subitem:last} For each $P \in \{P_+, P_-\}$ not contained in any element of $\PP$,
    add $P$ to $\PP$, and add all pairs $(i,\pm)$ such that $P_i \cap P$
    has positive volume to $\LL$.
    \end{enumerate}
  %\item Return to Step \ref{item:while}.
  \item Now that $\LL$ is empty, return $\sum_{P_i \in \PP} \vol P_i$.
  \end{enumerate}
\end{alg}

\begin{myremark}\label{rem:fiddle-with-polytopes}
  To determine whether $P_i \subseteq P_j$ in Step~\ref{item:initialize}
  and whether $P_\pm \subseteq P$ in
  Step~\ref{subitem:contained}, it suffices to determine whether
  $\vol P = \vol (P_\pm \cap P)$.  Since our polytopes are given by linear
  inequalities, we can intersect them simply by concatenating the lists of
  inequalities.  Magma handles these issues without difficulty.
\end{myremark}

\begin{thm} Algorithm \ref{alg:volume} terminates and returns the volume of
  $\cup_{i=1}^k P_i$.
\end{thm}

\begin{proof} To prove that the algorithm terminates, let $\Pi$ be the set of
  polytopes obtained by intersecting an element of $\PP_0$ with a collection
  of half-spaces determined by the facets of elements of $\PP_0$.
  The set $\Pi$ is finite and every polyhedron considered in the course
  of the algorithm belongs to $\Pi$.  Further, every pass through the while-loop
  (Steps \ref{subitem:first}--\ref{subitem:last})
  either increases the number of elements of $\PP$ (if neither $P_+$ nor
  $P_-$ is contained in any other element of $\PP$) or decreases the sum of
  volumes of pairwise overlaps (otherwise).  Again, the order of $\PP$ is
  bounded by $|\Pi|$
  and the set of possible sums of volumes of pairwise overlaps is
  finite, being a subset of the set of sums of subsets of the intersections
  of pairs of elements of $\Pi$.  Thus the algorithm must
  terminate.

  We now show that the algorithm is correct.
  It is clear that a pass through the loop cannot alter
  $\cup_{P_i \in \PP} P_i$, so the volume does not change, and that when $\LL$
  is empty the intersection of any two elements of $\PP$ has volume $0$.
  Thus by inclusion-exclusion the volume is $\sum_{P_i \in \PP} \vol P_i$ at
  that point.
\end{proof}    

\begin{myremark}\label{rem:not-so-bad}
  The upper bound on running time implied by an explicit version of the
  argument for termination given just above is
  horrifyingly large.  However, in
  practice the running time seems quite reasonable:
  see Section~\ref{sec:results-level-1} for some examples.
\end{myremark}

\section{Trace-minimal elements and reducers}\label{sec:trace-min-reducer}
In this section we present the geometric and combinatorial results that
underlie our estimates for the defect of the cusp singularities of a
Hilbert modular variety, in particular a threefold.
% do we need to think about elliptic singularities too?  See T-V, etc.
% not for fields other than Q(zeta_7)^+, Q(zeta_9)^+
We use three basic concepts in the geometry of the ring of integers in a
totally real number field, which we introduce here.

\begin{defn}\label{def:tmin-bal-red}
  Let $K$ be a totally real number field and let $d = [K:\Q]$;
  we write $x \gg 0$ to mean that $x$ is totally positive and we
  write $K_+$ for $\{x \in K: x \gg 0\}$.  Let
  $U_+ = \{x \in \O_K^\times: x \gg 0\}$ be the group of totally positive units
  of $K$.  Given $x \in K$, we say that $x$ is {\em trace-minimal} if
  $\tr ux \ge \tr x$ for all $u \in U_+$.

  Let $b \in \R^+$, and for $x \in K_+$, let $x_1, \dots, x_d$ be the real
  embeddings of $x$.  If $0 < \max_i x_i/\min_i x_i \le b$ we say that $x$ is
  {\em $b$-balanced}.

  Finally, if $I$ is a fractional ideal and
  $x \gg 0 \in I$ has the property that $\tr xy < \min (\Q^+ \cap I) \tr y$
  for some trace-minimal $y \in K$, then $x$ is a {\em reducer} relative
  to $I$, or an {\em $I$-reducer}.  We abbreviate $\min (\Q^+ \cap I)$ to
  $\min I$.
  %% If the inequality is strict we speak of
  %% a {\em strict $I$-reducer}.
  When $I = \O_K$ we refer simply to a {\em reducer}.  The set of
  $I$-reducers will be denoted $\RR_I$.  We also denote
  $\RR_I \cup \{\min I\}$ by $\RR_I'$.
  %%, and the set of strict $I$-reducers by $\RR_I'$.
\end{defn}

\begin{ex}\label{ex:reducer}
  In $\Q(\sqrt 7)$, the equality $(3-\sqrt 7)(3+\sqrt 7) = 2$ shows
  that $3 \pm \sqrt 7$ are reducers.  In Remark~\ref{rem:which-larger-fields}
  we showed that $\frac{9\pm\sqrt{65}}{2}$ are reducers in $\Q(\sqrt{65})$.
\end{ex}

\begin{myremark}\label{rem:reducer-basic}
  Not all ideals have reducers.  In particular,
  the key property of the fields considered by Thomas-Vasquez and Grundman in
  \cite{tv-rcs,gr-dcs,gr-ds} is that $(1)$ has no reducers.  In general,
  the larger the fundamental units of the number field, the less the units
  are capable of reducing its totally positive elements, and the more reducers
  it will have.  For example,
  in $\Q(\sqrt{46})$, where the fundamental unit is $24335 + 3588 \sqrt{46}$,
  there are $3542$ reducers.  Intuitively this is because
  a field with large fundamental units has a large fundamental domain for their
  action on the positive orthant, which allows for many elements of small trace
  but not too small norm which are divisible only by elements of larger trace.
  %% In this paper we
  %% do not need reducers that are not strict other than $\min I$.
  % maybe I should just get rid of them, but it seems more functorial this way.
  % I got rid of them, partly to save space.
  Note that if $r \in \Q^+$ then $\RR_{rI} = r\RR_I$. %and $\RR_{rI}' = r\RR_I'$.
\end{myremark}

%% \begin{ex}\label{ex:strict-reducer}
%%   The simplest example of a strict reducer is
%%   $x = 3 - \sqrt{7} \in \Q(\sqrt 7)$.  Let $y = 3 + \sqrt 7$: then
%%   $\tr xy = 4 < 6 = \tr y$.  To see that $y$ is trace-minimal, note that
%%   the fundamental unit is $8 + 3\sqrt 7$, and clearly
%%   $\tr((3+\sqrt 7)(8+3\sqrt 7)^k) \ge \tr(3+\sqrt 7) = 6$ for $k \ge 0$.  For
%%   $k < 0$ note that $(3+\sqrt 7)(8+3\sqrt 7)^k = (3-\sqrt 7)(8-3\sqrt 7)^{1-k}$,
%%   so again the trace is at least that of $3-\sqrt 7$, which is $6$.
%% \end{ex}

%% \begin{myremark} The cusp singularities for the
%%   cubic fields considered by Thomas-Vasquez and Grundman in
%%   \cite{tv-rcs,gr-dcs,gr-ds} 
%%   have particularly simple resolutions.  From
%%   our point of view these fields are distinguished by the property of having
%%   no strict reducers for narrowly principal ideals,
%%   which greatly simplifies the calculations needed to evaluate
%%   the defects.  It is possible that the fields they consider are the only cubic
%%   fields with this property, but we do not know how to prove this.
%% \end{myremark}

\subsection{Trace-minimal and balanced elements}\label{sec:trace-minimal-balanced}
We now prove some basic properties of these definitions.  Some of these
results, in particular Lemma~\ref{lem:trace-min-tot-pos},
Lemma~\ref{lem:trace-min-finite},
Corollary~\ref{cor:trace-min-finite-without-tp}, are essentially contained in
\cite[Lemma 3]{shintani}, and our proofs are closely related as well.

\begin{defn}\label{def:tot-pos-one}
  For this entire section, let us fix a set $S_\pm$ of units of $\O_K$ in
  bijection with the set
  of proper nonempty subsets of the set of real places of $K$ such that
  the unit $u_R$ corresponding to a subset $R$ is totally positive and
  greater than $1$ exactly at the places in $R$.
  In particular, for $1 \le i \le d$ let $\iu \in S_\pm$ be 
  such that $\iu_j > 1$ if and only if $j = i$ (recall that $\iu_j$ refers
  to the $j$th real embedding of $\iu$).
  In addition, let $S_d = \{v_1,\dots,v_d\}$, where $v_i$ is a unit that is
  totally positive and greater than $d$ in the $i$th real embedding.
\end{defn}

\begin{myremark}\label{rem:find-iu}
  The set $S_\pm$ can be found by standard lattice techniques.  In the
  log embedding of the units, we are looking for units with given negative
  coordinates.  One way to find such a unit is simply to enumerate
  units up to a given Euclidean norm in this embedding until the desired one
  is found.  The argument does not depend on the choice, but the bounds will
  be better if the $\iu_i$ are as small as possible.  Constructing $S_d$
  is easy.
\end{myremark}

\begin{lem}\label{lem:trace-min-tot-pos}
  A trace-minimal element is totally positive or $0$.
\end{lem}

\begin{proof} 
  Let $x \in K$.  If $x$ is totally negative, choose $i$
  such that $|x_i|$ is maximal; then $\tr(v_i x) < dx_i < \tr x$.
  If $x$ is positive in some real embeddings and negative in others,
  let $u_N$ be as in Definition~\ref{def:tot-pos-one}, where $N$ is the
  set of real embeddings where $x$ is negative.  Then $\tr u_Nx < \tr x$.
\end{proof}

\begin{notation}\label{not:bk} Let $b_K = \max_{i \ne j} \frac{\iu_i-1}{1-\iu_j}$.
\end{notation}

\begin{lem}\label{lem:trace-min-balanced}
  Every totally positive element of $K$ that is not $b_K$-balanced is
  reduced by one of the $\iu$.
\end{lem}

\begin{proof} Let $x \gg 0$ be trace-minimal.
  %By Lemma~\ref{lem:trace-min-tot-pos}, we know that $x \gg 0$.
  Suppose that $x_i$ is the smallest embedding of $x$ and $x_j$ is the largest.
  Then
  %% \begin{align*}
  %%   \tr x - \tr(\iu x) &= \sum_{k=1}^d (x - \iu x)_k \\
  %%   & \ge (x-\iu x)_i + (x-\iu x)_j\\
  %%   &= x_i(1-\iu_i) + x_j(1-\iu_j),\\
  %% \end{align*}
  \begin{equation*}
    \tr x - \tr(\iu x) = \sum_{k=1}^d (x - \iu x)_k
    \ge (x-\iu x)_i + (x-\iu x)_j
    = x_i(1-\iu_i) + x_j(1-\iu_j).
  \end{equation*}
  Now $b_K \ge \frac{\iu_i-1}{1-\iu_j}$, so if $\frac{x_j}{x_i} > b_K$ this is
  positive, and so $x$ is reduced by $\iu$.
\end{proof}

We now describe a finite set $S_K$ of units such that every totally positive
element that is not trace-minimal is reduced by one of them.  
\begin{lem}\label{lem:trace-min-finite}
  Let $b_K$ be the bound of Notation~\ref{not:bk}.  Let
  $S_K$ consist of the units $\iu$ of Definition~\ref{def:tot-pos-one},
  together with all totally positive units all of whose real embeddings
  are at most $db_K$.  If $x \gg 0 \in K$ satisfies $\tr ux \ge \tr x$
  for all $u \in S_K$, then $x$ is trace-minimal.
\end{lem}

\begin{proof} If $x$ is not $b_K$-balanced, then by
  Lemma~\ref{lem:trace-min-balanced} it is reduced by one of the $\iu$, which
  belongs to $S_K$.   Let $u$ be a totally
  positive unit with $u_i > db_K$ for some~$i$.  
  If $x$ is $b_K$-balanced, we then have
  $\tr ux \ge db_K x_i \ge db_K x_{\min{}} \ge dx_{\max{}} \ge \tr x$, where
  $x_{\min{}}, x_{\max{}}$ are the smallest and largest real embeddings of $x$,
  respectively.  It follows that if $x$ is $b_K$-balanced and not
  trace-minimal, then $x$ is reduced by a unit whose real embeddings are
  at most $db_K$ and which therefore belongs to $S_K$.
\end{proof}

We can find the finitely many totally positive
units all of whose real embeddings are at most $db_K$: indeed, in
the Minkowski embedding $x \to (\log |x_i|)$, these correspond to lattice
points within a compact subset of the hyperplane $\sum_i r_i = 0$.
Thus Algorithm~\ref{alg:int-points-rat-poly} applies.

\begin{cor}\label{cor:trace-min-finite-without-tp}
  There is a finite set $S'_K$ of totally positive units such that every
  element $x \in K$ (totally positive or not) that is not trace-minimal
  satisfies $\tr ux < \tr x$ for some $u \in S'_K$.
\end{cor}

\begin{proof} The sets $S_K$, $S_\pm$, and $S_d$
  are all finite.  We choose $S'_K$ to be their union.
\end{proof}

We extend the concept of trace-minimality to $K \otimes_\Q \R$.
If we use a basis for $K \otimes_\Q \R$ given by elements of $K$, then all of the
inequalities defining our polyhedra have coefficients in $\Q$, so
Algorithm~\ref{alg:int-points-rat-poly}
applies and we can always determine whether an integral point in
the enlarged rational polytope belongs to the original without the
possibility of error due to round-off.

\begin{cor}\label{cor:tm-is-polyhedral}
  The set of trace-minimal elements of $K \otimes \R$ is
  a polyhedral cone with finitely many rational faces.  Thus, for
  $q \in \Q^+$, the set of such elements with $\tr x \le q$ is a rational
  polytope.
\end{cor}

\begin{proof} Every element $u \in S'_K$
  (Corollary~\ref{cor:trace-min-finite-without-tp})
  defines a half-space by $\tr ux \ge \tr x$, and the intersection of these
  is the trace-minimal cone.
\end{proof}

\begin{myremark} In fact the trace-minimal cone is defined by
  $\tr ux \ge \tr x$ for $u \in S_K$.  To see this, note that the two cones
  have the same intersection with the totally positive cone.  However, every
  nonzero point of the trace-minimal cone is interior to the totally
  positive cone by Lemma~\ref{lem:trace-min-balanced}.  It follows that the
  two cones are equal.
\end{myremark}

\begin{myremark} Let $V \subseteq \otp$ be a subgroup of finite index,
  and say that $x \in K_+$ is {\em $V$-trace-minimal} if
  $\tr vx \ge \tr x$ for all $v \in V$.  The results and proofs of this
  section extend to the slightly more general situation where ``trace-minimal''
  is replaced by ``$V$-trace-minimal''; for simplicity we do
  not state these explicitly here.
\end{myremark}

We summarize the discussion of this section in an algorithm.
\begin{alg}\label{alg:trace-minimal}
  Let $V \subseteq \otp$ be a subgroup of finite index, where $K$ is a
  totally real number field.  Determine the
  $V$-trace-minimal cone.
  \begin{enumerate}
  \item For each real place $R_i$ of $K$, determine a unit $\iv \in V$ greater
    than $1$ at $R_i$ and less than $1$ at all other real places (see
    Remark~\ref{rem:find-iu}).  Let $U_V$ be the set of these
    units.  Calculate a constant
    $b_{K,V}$ as in Notation~\ref{not:bk}.
  \item Let $S_V$ be the subset of $V$ consisting of units less than or
    equal to $db_{K,V}$ at all real places.  Again, this is a standard lattice
    calculation.
  \item Return $\cap_{u \in S_V \cup U_V} H_u$, where $H_u$ is the half-space
    defined by $\tr xu \ge \tr x$.
  \end{enumerate}
\end{alg}

\subsection{Reducers}\label{sec:reducers}
We now consider the problem of determining the complete set of $I$-reducers in
a totally real number field.

\begin{lem}\label{lem:reducers-finite}
  Let $b_K$ be the bound of Notation~\ref{not:bk}.
  Every $I$-reducer is at most $(\min I) db_K$ at all real places.
\end{lem}

\begin{proof} This follows from the same argument that we used in
  Lemma~\ref{lem:trace-min-finite}.
  %%to show that a $b_K$-balanced element
  %%cannot be reduced by a unit that is greater than $db_K$ at any real place.
\end{proof}

Lemma \ref{lem:reducers-finite}
allows one to compute a finite subset of $K$ that contains
all the reducers via Algorithm~\ref{alg:int-points-rat-poly},
but it does not give a good bound, nor does it allow us
to determine whether an element is in fact a reducer.  Thus we analyze the
situation more closely.

\begin{lem}\label{lem:reduce-extremal-ray}
  Let $M_K$ be the trace-minimal cone inside $K \otimes \R$, and let its
  extremal rays be $\R^+ v_i$ for $1 \le i \le m$, where $v_i \in K$.
  An element $x \in I$ is an $I$-reducer if and only if
  $\tr xv_i < \min I \tr v_i$ for some $i$.
\end{lem}

\begin{proof}
  The $v_i$ are totally positive, so ``if'' is just the
  definition.  For ``only if'', suppose that $\tr xy < \min I \tr y$ with $y$
  trace-minimal; by definition we have $y = \sum_i c_i v_i$ with $c_i \ge 0$.
  Thus $\tr xy = \sum_i c_i \tr xv_i$.  If $\tr xv_i \ge \min I \tr v_i$
  for all $i$, then
  \begin{equation*}
    \min I \tr y > \tr xy = \sum_i c_i \tr xv_i \ge \min I \sum_c c_i \tr v_i = \min I \tr y.
  \end{equation*}
  This contradiction establishes ``only if''.
\end{proof}

\begin{lem}\label{lem:reducers-simplex}
  Given $v \gg 0$, the elements $x \gg 0 \in I$ with $\tr xv < \min I \tr v$
  are naturally in correspondence with the integral points of a simplex.
\end{lem}

\begin{proof} The set of such $x$ is the set of integral points of the
  region in $\R^d$ defined by $x_i \ge 0$ for all $i$ and
  $\sum_{i=1}^d (x_i-\min I)v_i = 0$ that belong to the sublattice of $\Z^d$ defined
  by $I$.  By changing coordinates we convert this sublattice to the standard
  one.
\end{proof}

Thus we may use Algorithm~\ref{alg:int-points-rat-poly} to list the reducers
efficiently.  Again we summarize the discussion in an algorithm.

\begin{alg}\label{alg:reducers}
  Given a fractional ideal $M$, determine the $M$-reducers.
  \begin{enumerate}
  \item Determine the trace-minimal cone (Algorithm~\ref{alg:trace-minimal}).
  \item For each extremal ray $v\R^+$ of the trace-minimal cone, choose a
    representative $v \in K$ and determine the set
    $R_v = \{x \in M_+: \tr xv < \min M \tr v\}$
    %(strict inequality for strict reducers)
    as in Lemma~\ref{lem:reducers-simplex}.
  \item The answer is $\cup_v R_v$.
  \end{enumerate}
\end{alg}

%% \subsection{Level structure}
%% In addition to Hilbert modular varieties at level $1$, we would also like
%% to study those of higher level with geometric genus $0$ or $1$, as in
%% \cite{hamahata,hmf-mult,logan} among others.

%% \begin{prop}\label{prop:grundman-general-type}
%% \end{prop}

\section{Algorithms for defects}\label{sec:defect-algs}
We now consider two closely related problems concerned with counting elements
with multiples of bounded trace.  
Thus let $M$ be a lattice in $K$ and
$V \subseteq \otp$ be a subgroup of finite index preserving
$M$ (by Lemma~\ref{lem:group-preserved} such a subgroup exists).
We recall the notation from Definition~\ref{def:knoller-notation}
and the fact \cite[Satz 2.4]{knoller} that the $q$th defect of a cusp of type
$(M,V)$ is $\delta_{M,V}(q) = \#\Lambda_q(M)/V$.
Our problems are as follows:

\begin{problem}\label{prob:compute-one-delta}
  Given $M, V, q$, compute $\delta_{M,V}(q)$.
\end{problem}

\begin{problem}\label{prob:compute-all-delta}
  Given $M, V$, give an asymptotic formula for $\delta_{M,V}(q)$.
\end{problem}

For the first of these problems, 
most of the ideas of Section~\ref{sec:trace-min-reducer}
are not necessary.  By rescaling
we may assume that $M \subseteq \O_K$.  For simplicity, and because it is
the only case needed in this paper by \cite[Proposition 3.3.8]{hmf-mult},
we assume that $M$ is an ideal of $\O_K$.

\begin{lem}\label{lem:which-m}
  Let $\dif_K$ be the different ideal of $\O_K$.  Let $t \in \dif^{-1}_K$ and
  let $x \ne 0 \in \hat M$.  Define $I_1, I_2$ by
  $(x) = \dif^{-1}_KM^{-1} I_1$, $(t) = \dif^{-1}_K I_2$.
  Then there exists $m \in M$ with $xm = t$ if and only if $I_1|I_2$.
\end{lem}

\begin{proof} Of course $m \in K$ with $xm = t$ is unique, so we need only
  determine whether $x^{-1} t \in M$.  Since $I_1, I_2$ are integral
  and $(x^{-1}t) = MI_2I_1^{-1}$, the result follows.
\end{proof}

\begin{lem}\label{lem:defect-of-subgroup}
  For all ideals $M$, all $q > 0$, and all $V$ of finite index in $\otp$,
  we have $\delta_{M,V}(q)-1 = [\otp:V] \left(\delta_{M,\otp}-1\right)$.
\end{lem}

\begin{proof}
  The quotient map
  $(\Lambda_q(M) \setminus \{0\})/V \to (\Lambda_q(M) \setminus \{0\})/\otp$
  is $[\otp:V]$-to-$1$.
\end{proof}

Thus we obtain an algorithm to calculate the defects as follows:
\begin{alg}\label{alg:defect}
  Given $M, V$ where $M$ is an ideal, compute the first $q$ defects
  $(\delta_{M,V}(i))_{i=1}^q$.
  \begin{enumerate}
  \item List all totally positive
    elements of $\dif^{-1}_K$ with trace less than $q$, using
    Algorithm~\ref{alg:int-points-rat-poly}.  (We proved in
    Corollary~\ref{cor:tm-is-polyhedral} that these are the points of a lattice
    inside a rational polytope, so the algorithm is applicable.)
  \item Determine the sets of ideals
    $\isi_i = \{(t\dif_K): t \in \dif^{-1}_K, \tr t < i\}$ for $1 \le i \le q$.
  \item For each $I \in \isi_q$, determine the set of divisors
    $D_I$ of $I$, and let
    $D_{I,+} = \{J \in D_I: IJ^{-1}M \textup{ is narrowly principal}\}$.
    For each $i$
    let $\isd_i = \cup_{I \in \isi_i} D_{I,+}$.
  \item The answer is $([\otp:V]\#\isd_i + 1)_{i=1}^q$.
  \end{enumerate}
\end{alg}

\begin{prop}\label{prop:alg-defect-works}
  Algorithm~\ref{alg:defect} terminates and is correct.
\end{prop}

\begin{proof}
  Termination is immediate, since this algorithm has
  no loops and all of the steps are effective.  For correctness,
  first we assume that $V = \otp$.  Let $t \in \dif^{-1}_{K,+}$ with
  $\tr m < q$ and let $x \in {\hat M}_+$; define $I_1, I_2$ as in
  Lemma~\ref{lem:which-m}.  If $tx^{-1} \in M$ then, from the above,
  $I_1 I_2^{-1} = t^{-1}x M^{-1}$ and $I_1I_2^{-1}M$ is narrowly principal,
  being generated by $t^{-1}x$.  The converse follows similarly.  Counting
  narrowly principal ideals is the same as counting totally positive
  generators of those ideals up to totally positive units, so the result
  follows.  The last step, giving the answer,
  is justified by Lemma~\ref{lem:defect-of-subgroup}.
\end{proof}

\begin{myremark}\label{rem:avoid-repetition}
  We give the output in this form because computing the $q$th defect is not
  significantly harder than computing the first $q$ defects.
\end{myremark}

We now consider the problem of determining an asymptotic.

\begin{defn}\label{def:spiky}
  %Let $r$ be an $M$-reducer (Definition~\ref{def:tmin-bal-red}).
  For $r \gg 0 \in K$, 
  let $T_{M,V,r}(q)$  be the intersection of the $V$-trace-minimal cone
  with the half-space $\tr rx \le q$.  Let
  $T_{M,V}(q) = \cup_{r \in \RR_M'} T_{M,V,r}(q)$.
\end{defn}

\begin{ex}\label{ex:spiky-small}
  We illustrate this definition for a quadratic and a cubic field.
  First let $K = \Q(\sqrt{14})$.  The trace-minimal cone is bounded by
  rays through $4 \pm \sqrt{14}$, and the reducers are $4 \pm \sqrt{14}$.
  Thus the set $\RR_{\O_K}'$ consists of $3$ elements, the corresponding
  triangles being shown in the first plot of Figure~\ref{fig:spiky}.
  Next we consider the smallest cubic field that has reducers, namely
  $K = \cf{148}$, obtained by adjoining a root $\alpha$ of $x^3-x^2-3x+1$.
  The maximal order is $\Z[\alpha]$ and so we specify an element of $K$
  as a triple $(a,b,c)$ representing $a+b\alpha+c\alpha^2$.
  The trace-minimal cone is bounded by rays through
  \begin{equation*}
    (4,26,21),(15,-32,14),(17,18,6),(53,-44,10),(68,-2,-13),(139,34,-38)
  \end{equation*}
  (one verifies that for each $x$ in this list there are at least
  two totally positive units $u_x \ne 1$ with $\tr x = \tr xu_x$, an obvious
  necessary condition); we find $7$ reducers, namely
  \begin{equation*}
    (0,1,1),(1,-2,1),(1,2,1),(2,-7,3),(2,-3,1),(4,1,-1),(5,0,-1).
  \end{equation*}
  Thus $T_{\O_K,\otp}(1)$ is a union of $8$ polyhedra, shown in the second plot
  of Figure~\ref{fig:spiky}.  The largest polyhedron, containing a
  factor of $256/275$ of the volume of the union, corresponds to
  $T_{\O_K,\otp,1}(1)$ and is shown in yellow.  A Jupyter notebook in \cite{code}
  contains an interactive version of this plot.

  \begin{figure}
    \includegraphics[scale=0.7]{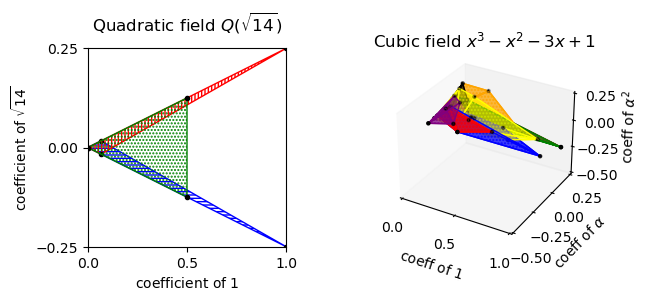}
    \caption{Regions describing elements of the quadratic and cubic fields of
      discriminant $56$ and $148$ respectively having an integral multiple
      with bounded trace.  For the quadratic field we have a union of $3$
      triangles $T_1, T_2, T_3$, corresponding to the reducers
      $4 - \sqrt{14}, 4 + \sqrt{14}$ and $1$, and shown in the figure in
      red, blue, and green respectively.  The vertices of $T_1$ are
      $(0,0),(1,-1/4),(1/15,1/60)$, and those of $T_3$ are
      $(0,0),(1/2,-1/8),(1/2,1/8)$.  We obtain $T_2$ by reflecting $T_1$ in the
      $y$-axis.  For the cubic field, the region is a union of $8$ polyhedra,
      each with $7$ vertices.}\label{fig:spiky}
  \end{figure}
\end{ex}

\begin{lem}\label{lem:scale-polytope}
  $T_{M,V}(1)$ is a finite union of rational polytopes and
  $T_{M,V}(q)$ is obtained by scaling $T_{M,V}(1)$ by $q$.  \qed
\end{lem}

\begin{lem}\label{lem:reducers-reduce}
  Let $x \in {\hat M}_+$ and suppose that $y \in M_+$ is
  such that $\tr xy < q$.  Then there is $r \in \RR_M'$
  such that $\tr xr < q$.
\end{lem}

\begin{proof}
  Suppose that $\tr xy < q$.  If $y \in \RR_M$ we take $r = y$.  If not,
  then $\tr x < q/\min M$ and we take $r = \min M$.
\end{proof}

\begin{lem}\label{lem:lattice-points-match}
  There is a surjective map from the set of $\hat M$-points of
  $T_{M,V}(q)$ to $\Lambda_{q+1}(M)/V$, taking $x \in \hat M$ to $xV$.
  If $x \in \hat M$ is in the interior of $T_{M,V}(q)$, then no
  other element of the set has the same image as $x$.
\end{lem}
%% Every element of $\Lambda_M(q)/V$ is represented by a point of
%% the lattice $\hat M$ in $T_{M,V}(q)$.
%% If the lattice point is in the interior, then
%% the representation is unique.

\begin{proof}
  First, this is indeed a map of the given sets: for $x \in \hat M$,
  being in $T_{M,V}(q)$ implies that there exists
  $y$ with $\tr xy \le q$ which is either a reducer or $\min M$,
  so $xV \subset \Lambda_{q+1}(M)$.
  For surjectivity, fix $xV \in \Lambda_{q+1}(M)/V$ and choose $x$ to be
  a trace-minimal element of $xV$.  By definition we have $\tr xy \le q$
  for some $y \in M$, and by Lemma~\ref{lem:reducers-reduce} we may take
  $y \in \RR_M'$.  Thus $x \in T_{M,V,y}(q) \subseteq T_{M,V}(q)$.
    
  For the injectivity statement, suppose that $x, x' \in \hat M \cap T_{M,V}(q)$
  have the same image in $\Lambda_q(M)/V$.  Since $x^{-1}x' \in V$ and
  $x, x'$ are trace-minimal, we must have $\tr x = \tr x'$, which means that
  $x, x'$ are on the boundary of the trace-minimal cone and hence of
  $T_{M,V}(q)$.
\end{proof}

\begin{alg}\label{alg:asymptotic-dth-power}
  Given $M, V$, where $M$ is a fractional ideal of $K$ and $V$ is a subgroup
  of finite index in $\otp$,
  compute the rational constant $c_{M,V}$ such that
  $\delta_{M,V}(q) \sim c_{M,V} q^d$.
  \begin{enumerate}
  \item Determine the trace-minimal cone by means of
    Algorithm~\ref{alg:trace-minimal}.
  \item Determine the reducers $\RR_M$ using
    Algorithm~\ref{alg:reducers}.
  \item Let $\PP$ be the set of polyhedra given by intersecting the
    trace-minimal cone with the half-spaces $\tr rx \le \min M \tr x$
    for each element of $\RR_M'$.  Use
    Algorithm~\ref{alg:volume} to determine the volume $E$
    of $\cup_{P \in \PP}P$.
  \item Return $E [\otp:V]/N_{K/\Q}(\hat M)$, where $N_{K/\Q}$ is the norm map
    on fractional ideals.
  \end{enumerate}
\end{alg}

%% \begin{prop}\label{prop:asymptotic-dth-power}
%%   There exists a rational
%%   constant $c_{M,V}$ such that $\delta_{M,V}(q) \sim c_{M,V} q^d$.
%% \end{prop}

\begin{prop}\label{prop:asymptotic-dth-power-works}
  Algorithm~\ref{alg:asymptotic-dth-power} terminates and is correct.
\end{prop}

\begin{proof}
  As before, termination is immediate from the termination of each step, since
  there are no loops.  For correctness, by
  combining Lemmas~\ref{lem:scale-polytope}, \ref{lem:lattice-points-match}
  we see that asymptotically the elements of $\Lambda_q(M)/\otp$ are
  in bijection with the points of $\hat M$ lying in $\cup_{P \in \PP} P$
  scaled by $q$.  (The failure of injectivity
  only affects the points on polytopes of lower dimension, which are
  asymptotically negligible.)  The correctness of the algorithm follows
  for $V = \otp$, since $\covol \hat M/\covol \O_K = N(\hat M)$.
  The more general result follows from Lemma~\ref{lem:defect-of-subgroup}.
  The desired asymptotic
  is then a standard fact \cite[Lemma 3.19]{beck-robins},
  and since $\delta_{M,V}(q) = \#\Lambda_q(M)/V$ the result follows.
\end{proof}

\begin{myremark}
  In this algorithm we could replace the trace-minimal cone
  by the $V$-trace-minimal cone and omit the multiplication by $[\otp:V]$
  in the last step.
\end{myremark}

\begin{cor}\label{cor:change-m-v}
  The constant $c_{M,V}$ depends only on the narrow ideal class of $M$, 
  and $c_{M,V} = [\otp:V] c_{M,\otp}$.
\end{cor}

\begin{proof} 
  Let $I$ be narrowly principal and generated by $p$.  There is an obvious
  bijection $\Lambda_q(MI)/V \leftrightarrow \Lambda_q(M)/V$ given by
  multiplication by $p$.  For the second statement, if we
  pass to a subgroup of $V$ of index $n$, the volume of the
  fundamental domain is multiplied by $n$.  The result follows.
\end{proof}

\begin{myremark}\label{rem:choice-of-m}
  Although $c_{pM,V} = c_{M,V}$ for $p \gg 0$, the time required
  to determine it using Algorithm~\ref{alg:asymptotic-dth-power} seems
  to grow rapidly with $[pM:(\min(pM))]$, and it is wise to choose $p$
  so as to minimize this quantity.
\end{myremark}

We have now solved Problems~\ref{prob:compute-one-delta},
\ref{prob:compute-all-delta}.
%% we may assume that $V = \otp$.
%% In this case Algorithm~\ref{alg:volume} can be used to compute $\vol(T_V(1))$,
%% so we now have a method to solve these problems for a given field by means of
%% a finite computation.

%% could get rid of this
\begin{myremark}\label{rem:not-independent}
  Even in dimension $2$,
  it is not true that $c_{M,V}$ is independent of $M$.  For
  example, take $K = \Q(\sqrt{3})$ and consider the two ideals
  $(1), (\sqrt 3)$ that represent the narrow class group.
  If $M = (1)$ we find that
  $\delta_M(q) = 1 + q(q-1)$, while for $M' = (\sqrt 3)$ it turns out
  that $\delta_{M'}(q) = 1 + q(q-1)/2$.
  By \cite{hmf} or \cite[Example II.5.1]{vdg}
  the self-intersections of the curves in the cusp resolutions are
  respectively $-4$ and $-3,-2$, so this is in accordance with
  \cite[Proposition III.3.6]{vdg}.  (Note that \cite{vdg} refers to the
  obstruction to extending forms that vanish at the cusp across a resolution,
  not all forms, so the quantity considered by van der Geer is $1$ less than
  our defects.)  In general it appears that, as $I$ ranges over genus
  representatives, the largest $c_{I,\otp}$ occurs for the principal
  genus.  For examples in dimension $3$, compare Tables \ref{tab:cubic-pg-0}
  and \ref{tab:cubic-nonprincipal-g0}.
\end{myremark}

\begin{myremark}\label{rem:answer-simple}
  By a standard result on rational polytopes, first proved by
  Ehrhart in 1962 but more easily accessible as
  \cite[Theorem 3.23]{beck-robins}, the $\delta_{M,V}(q)$
  for a given cusp are polynomial on residue classes.  However,
  computations of the $\delta_{M,V}(q)$ indicate that they satisfy a simpler
  formula than might be expected, especially in the case
  $M = \O_K, V = \otp$.  See Remark~\ref{rem:simple-defect}.
  We therefore suspect that there is some further structure to the
  $\delta(q)$ that remains to be elucidated.
\end{myremark}

\section{Results at level $1$}\label{sec:results-level-1}

In this section we will describe the application of the algorithms presented
here to Hilbert modular threefolds of level $1$.
To do so we need asymptotic formulas for both the dimension of the
space of modular forms of weight $2k$ and the defects.  For the first of
these, we already have the result in Proposition~\ref{prop:asymptotic-dim-hmf}.
We easily derive a useful consequence:

\begin{prop}\label{prop:general-type}
  Suppose that $(-1)^d \cdot 2\zeta_K(-1) > \sum_i c_i$, where the sum
  ranges over the cusps of $H_K$ and the $q$th defect of the $i$th cusp is
  asymptotic to $c_i q^d$.  Then $H_K$ is of general type.
\end{prop}

\begin{proof}
  We substitute the asymptotics for $\dim M_{2q}$ and $\sum_P \delta_{P}(q)$
  into (\ref{eqn:lower-bound}).  The assumptions imply that
  $\hhh^0(qK)$ is not $o(q^d)$ and so $H_K$ is of general type.
\end{proof}

Thus we will apply Algorithm~\ref{alg:asymptotic-dth-power} to determine
the asymptotic rate of growth $c_2k^3$ of the $k$th defect and compare it to
$-2\zeta_K(-1)$.  In particular, we recall that if $-2\zeta_K(-1) > c_2$ then
$H_K$ is of general type (see Remark~\ref{rem:why-probably-enough}).
% We could start with Q(sqrt 14), and for a talk it's probably a good idea.
We begin by reexamining one of the 
threefolds of arithmetic genus $1$ proved by Grundman \cite{gr-ds} to
have positive Kodaira dimension.

\begin{ex}\label{ex:gr-761}
  Let $K = K_{2,5}$ be the field obtained by adjoining a root $\alpha$
  of $x(x-2)(x-5)-1$.  As in \cite{gr-ds}, this field has discriminant
  $761$, the ring of integers is generated by $\alpha$, and the unit group is
  generated by $t,t-2,-1$, with the totally positive units generated by
  $t, (t-2)^2$.  The class number is $1$, so there is only one cusp up to
  equivalence, and the narrow class number is $2$.
  We consider the variety $H_K$ corresponding to the principal genus.
  
  % nonsense
  %% Because the narrow class number is $2$, the Hilbert modular variety naturally
  %% has two components.  By Corollary~\ref{cor:change-m-v}, the asymptotic
  %% for the defects is the same for both; so too is the asymptotic for the
  %% dimension of the space of modular forms.  Thus in proving general type we
  %% can ignore the distinction between the two components.

  The bound $b_K$ of Lemma~\ref{lem:trace-min-balanced} can be
  taken to be $76$.  Every totally positive element that is not
  trace-minimal is reduced by $(t-2)^it^j$ for some $(i,j) \in S$, where
  \begin{align*}
    S = &\{(-2,-3),(-2,-2),(-2,-1),(-2,0),(-2,1),(-2,2),(0,-2),\\
    &\quad (0,-1),(0,1),(0,2),(0,3),(2,-1),(2,0),(2,1),(4,0)\}.
  \end{align*}

  There are no reducers for $(1)$,
  so determining the trace-minimal cone is
  enough to calculate the defects.  The rays defining this cone are spanned by
  \begin{align*}
    &(23,29,51),\quad (350,-55,-18),\quad (682,132,-109),\\
    &\quad (-35,-375,154),\quad (84,139,87), \quad(1179,130,-165),
  \end{align*}
  where $(a,b,c)$ abbreviates $a+b\alpha+c\alpha^2$.

  Cutting the cone by the hyperplane $\tr x = 1$, we obtain a polyhedron whose
  volume is $13/4$.  This matches Grundman's formula \cite[Theorem 1]{gr-ds},
  from which it follows that the $q$th defect of the cusp defined by the
  group of totally positive units is asymptotic to $13q^3/4$.  Thus the
  defect for the group of squares of units is asymptotic to $13q^3/2$.
  Since $-2\zeta_K(-1) = 20/3 > 13/2$,
  Proposition~\ref{prop:general-type} implies that
  principal component of the Hilbert modular threefold for $K$ is of
  general type.  On the author's laptop 
  (a modest computational resource by the standards of the year $2025$),
  this computation takes only $0.5$ seconds.
\end{ex}

\begin{myremark}\label{rem:gr-985}
  Similar considerations apply to $K_{3,5} =  \cf{985}$,
  giving the result that the principal component of $H_K$ is of general
  type.  Though Grundman only states that at least one plurigenus of each
  of these varieties is positive, we consider the statement that they are
  of general type to be implicit in her work.
\end{myremark}

\begin{ex}\label{ex:disc-473}
  We now consider a more involved example: the cubic field \cf{473}.
  It is generated by a root $\alpha$ of $t^3 - 5t - 1$.  Then
  $\O_K = \Z[\alpha]$, while $\O^\times_K = \langle -1,\alpha,\alpha+2\rangle$
  and $\otp = {\O^\times_K}^2$.  We have $h_K = h^+_K = 1$.
  Every non-trace-minimal element is reduced by $\alpha^i(\alpha+2)^j$
  for some $(i,j) \in S$, where
  \begin{equation*}
    S = \{(-4,-2),(-2,-2),(-2,0),(-2,2),(0,-2),(0,2),(2,-2),(2,0),(2,2),(4,0)\}.
  \end{equation*}
  It turns out that there are $19$ reducers,
  whose norms range from $3$ to $15$.
  We find that $8$ of the $20$ polyhedra that correspond to the elements of
  $\RR_{\O_K}'$ are redundant, so we need only find the
  volume of a union of $12$ polyhedra.  This is small enough that we can
  check the result of Algorithm~\ref{alg:volume}
  by inclusion-exclusion, finding by both methods that the
  volume is $79/24\disc_K$.  On the other hand, we have 
  $-2\zeta_K(-1) = 10/3$.  Since $10/3-79/24 > 0$,
  this proves that $H_K$ is of general type.  
  This example takes under $4$ seconds, most of which is used for the volume
  computation, which passes through the main loop of Algorithm~\ref{alg:volume}
  (step~(\ref{item:while})) $22$ times.
\end{ex}

We now survey the fields for which $p_g(H_K) \le 1$ \cite[Table I]{gl},
beginning with those for which we cannot prove $H_{K;I}$ to be of general
type.  (Although we do not have a real result for these, the information
shown here will be useful in Section~\ref{sec:results-higher-level}.)
In the tables in this section, we use the following notation.
The columns labeled $h^+, r, n, t, t'$ refer to
the narrow class number, the number of reducers,
the number of passes through the main loop of
Algorithm~\ref{alg:volume}, the time for the whole calculation, and the time
taken by one particular run of Algorithm~\ref{alg:volume}.  We omit the
class number because it is always $1$ for cubic fields $K$ with
$p_g(H_K) \le 1$ (for this reason it is also unnecessary to specify the cusp),
and we omit $h^+$ in tables for the
nonprincipal genus, since it is always $2$ there.  The exceptional
speed of the examples of discriminant
$49, 81, 169, 229, 257, 361, 697, 761, 985$,
already considered by
Thomas-Vasquez or Grundman \cite{tv-rcs,gr-dcs,gr-ds}, reflects that these are
fields with no reducers and a very simple cusp resolution.  The field
\cf{1489} is also $K_{1,8}$.  Likewise, in the
genus-$1$ case, \cf{1369} is the cubic subfield of
$\Q(\zeta_{37})$, which is $K_7$ (Definition~\ref{def:krs-kn})
and has no reducers.  On the other hand, although
\cf{1765}, for example, can be defined by the
special polynomial $p(t) = t(t-2)(t-42)-1$, the maximal order is not generated
by a root of $p$ and the unit group is not generated by $t, t-2$, so this does
not imply the existence of a particularly simple resolution.

%% Note that all cubic fields
%% for which $p_g(H_K) \le 1$ have class number $1$, so we need only consider one
%% cusp.  This is false for
%% quadratic fields, so we do not expect a proof that
%% does not use an enumeration of such fields.

\begin{prop}\label{prop:small-not-gt}
  For the first $11$ cubic fields as ordered by discriminant,
  the zeta values and asymptotic growth of defects coming from the cusp
  resolution are as shown in Table~\ref{tab:cubic-not-gt}.
\end{prop}

\begin{proof} We simply apply Algorithm~\ref{alg:volume} to compute the
  growth of defects.
\end{proof}

\begin{table}[h!]
  \caption{Hilbert modular threefolds for fields of discriminant at most
    $469$, not shown to be of general type.}
  %For discriminant $229$ both
  %  genera are given; the star indicates the nonprincipal genus.}
  \begin{tabular}{|l|c|c|c|c|c|c|c|c|}
  \hline
  $K$&$-2\zeta_K(-1)$&$c_{\O_K,\otp}$&$h^+$&$r$&$n$&$t$&$t'$\\
  \hline
$\cf{49}$&$2/21$&$5/12$&$1$&$0$&$0$&$0.250$&$0.010$ \\ \hline 
$\cf{81}$&$2/9$&$3/4$&$1$&$0$&$0$&$0.130$&$0.000$ \\ \hline 
$\cf{148}$&$2/3$&$55/36$&$1$&$7$&$8$&$0.950$&$0.700$ \\ \hline 
$\cf{169}$&$2/3$&$17/12$&$1$&$0$&$0$&$0.190$&$0.000$ \\ \hline 
$\cf{229}$&$4/3$&$3$&$2$&$0$&$0$&$0.240$&$0.000$ \\ \hline 
%$\cf{229}^*$&$4/3$&$14/9$&$2$&$8$&$9$&$1.340$&$0.960$ \\ \hline 
$\cf{257}$&$4/3$&$26/9$&$2$&$0$&$0$&$0.200$&$0.000$ \\ \hline 
$\cf{316}$&$8/3$&$137/36$&$1$&$60$&$35$&$7.160$&$5.500$ \\ \hline 
$\cf{321}$&$2$&$8/3$&$1$&$17$&$21$&$3.050$&$2.510$ \\ \hline 
$\cf{361}$&$2$&$29/12$&$1$&$0$&$0$&$0.390$&$0.000$ \\ \hline 
$\cf{404}$&$10/3$&$143/36$&$1$&$94$&$55$&$8.640$&$6.920$ \\ \hline 
$\cf{469}$&$4$&$49/12$&$1$&$58$&$25$&$6.330$&$3.990$ \\ \hline 
  \end{tabular}\label{tab:cubic-not-gt}
\end{table}

\begin{thm}\label{thm:general-type-from-473}
  Let $K$ be a cubic field of discriminant at least $473$ and narrow
  class number $1$ such that the
  geometric genus of $H_K$ is at most $1$.  Then the principal component of
  $H_K$ and $\hat H_K$ is of general type, unless possibly $D(K) = 697$
  or $788$.
\end{thm}

\begin{proof}
  This is computed by the method of Example~\ref{ex:disc-473}.
  When $h^+ = 1$, there is no difference between $H_K$ and $\hat H_K$.
  When $h^+ = 2$, the dimension of the space of cusp forms of weight
  $2k$ for $\PSL_2$ is asymptotically twice that for $\PGL_2^+$, but
  the stabilizer of the cusp for $\PSL_2$ is of index $2$ in that for
  $\PGL_2^+$, so the defect is also asymptotically twice as large.
  The elliptic points do not contribute by Remark~\ref{rem:cqs-2-3-canonical}.
  Thus the calculation is the same for the two groups.  
  Tables \ref{tab:cubic-pg-0}, \ref{tab:cubic-pg-1}
  show the results of our computations for the $22+14 = 36$ real
  cubic fields satisfying the hypotheses of the theorem.
  The general type result
  follows by noting that the second column, the constant in the asymptotic
  for the dimension of the space of modular forms, is greater than the third,
  the constant in the asymptotic for the defect.
\end{proof}

\begin{table}[h!]
  \caption{Hilbert modular threefolds of general type, except for discriminant
    $697$ and $788$, and geometric genus $0$
    (principal genus).}
  \begin{tabular}{|l|c|c|c|c|c|c|c|c|}
  \hline
  $K$&$-2\zeta_K(-1)$&$c_{\O_K,\otp}$&$h^+$&$r$&$n$&$t$&$t'$\\
  \hline
  $\cf{473}$&$10/3$&$79/24$&$1$&$19$&$22$&$3.750$&$3.020$ \\ \hline 
  $\cf{564}$&$6$&$1021/180$&$1$&$227$&$84$&$25.230$&$20.530$ \\ \hline 
  $\cf{568}$&$20/3$&$1141/180$&$1$&$477$&$60$&$28.130$&$17.700$ \\ \hline 
  $\cf{621}$&$20/3$&$413/72$&$1$&$219$&$95$&$34.110$&$28.080$ \\ \hline 
  $\cf{697}^*$&$16/3$&$67/12$&$2$&$0$&$0$&$0.330$&$0.010$ \\ \hline 
  $\cf{733}$&$8$&$221/36$&$1$&$251$&$75$&$30.510$&$22.020$ \\ \hline 
  $\cf{756}$&$26/3$&$469/72$&$1$&$297$&$102$&$55.440$&$32.830$ \\ \hline 
  $\cf{761}$&$20/3$&$13/2$&$2$&$0$&$0$&$0.510$&$0.010$ \\ \hline 
  $\cf{785}$&$22/3$&$29/6$&$1$&$71$&$76$&$19.080$&$16.820$ \\ \hline 
  $\cf{788}^*$&$28/3$&$113/12$&$2$&$50$&$16$&$7.850$&$2.930$ \\ \hline 
  $\cf{837}$&$32/3$&$1381/180$&$1$&$1154$&$95$&$75.750$&$41.770$ \\ \hline 
  $\cf{892}$&$40/3$&$1091/90$&$2$&$141$&$54$&$20.220$&$13.610$ \\ \hline 
  $\cf{940}$&$44/3$&$703/72$&$1$&$1209$&$194$&$154.230$&$117.120$ \\ \hline 
  $\cf{985}$&$28/3$&$70/9$&$2$&$0$&$0$&$0.660$&$0.010$ \\ \hline 
  $\cf{993}$&$34/3$&$2377/360$&$1$&$120$&$68$&$35.030$&$27.330$ \\ \hline 
  $\cf{1076}$&$44/3$&$719/60$&$2$&$28$&$11$&$5.590$&$3.580$ \\ \hline 
  $\cf{1257}$&$16$&$101/9$&$2$&$15$&$15$&$5.880$&$4.540$ \\ \hline 
  $\cf{1300}$&$18$&$629/72$&$1$&$742$&$94$&$102.790$&$59.860$ \\ \hline 
  $\cf{1345}$&$46/3$&$779/120$&$1$&$95$&$143$&$55.150$&$50.250$ \\ \hline 
  $\cf{1396}$&$64/3$&$535/36$&$2$&$45$&$19$&$9.950$&$5.140$ \\ \hline 
  $\cf{1489}$&$16$&$134/15$&$2$&$0$&$0$&$0.920$&$0.000$ \\ \hline 
  $\cf{1593}$&$64/3$&$253/20$&$2$&$7$&$8$&$3.100$&$1.520$ \\ \hline 
  \end{tabular}\label{tab:cubic-pg-0}
\end{table}

\begin{table}[h!]
  \caption{Hilbert modular threefolds and geometric genus $1$
    (principal genus), all of general type.}
  \begin{tabular}{|l|c|c|c|c|c|c|c|c|}
  \hline
  $K$&$-2\zeta_K(-1)$&$c_{\O_K,\otp}$&$h^+$&$r$&$n$&$t$&$t'$\\
  \hline
  $\cf{1016}$&$52/3$&$124/9$&$2$&$842$&$70$&$85.090$&$29.070$ \\ \hline 
  $\cf{1101}$&$52/3$&$247/24$&$1$&$3198$&$163$&$278.630$&$125.560$ \\ \hline 
  $\cf{1129}$&$44/3$&$109/10$&$2$&$41$&$27$&$12.080$&$9.490$ \\ \hline 
  $\cf{1229}$&$56/3$&$37/3$&$2$&$64$&$43$&$20.580$&$17.330$ \\ \hline 
  $\cf{1369}$&$14$&$65/12$&$1$&$0$&$0$&$1.300$&$0.000$ \\ \hline 
  $\cf{1373}$&$68/3$&$1861/168$&$1$&$1351$&$171$&$173.030$&$109.780$ \\ \hline 
  $\cf{1425}$&$58/3$&$1519/180$&$1$&$481$&$145$&$97.420$&$72.400$ \\ \hline 
  $\cf{1492}$&$68/3$&$667/45$&$2$&$28$&$17$&$7.980$&$5.880$ \\ \hline 
  $\cf{1573}$&$76/3$&$2671/252$&$1$&$1088$&$230$&$234.060$&$179.500$ \\ \hline 
  $\cf{1620}$&$86/3$&$37/3$&$1$&$4708$&$274$&$623.380$&$300.940$ \\ \hline 
  $\cf{1765}$&$92/3$&$658/45$&$2$&$908$&$68$&$238.690$&$36.970$ \\ \hline 
  $\cf{1825}$&$68/3$&$401/36$&$2$&$3$&$4$&$2.810$&$1.230$ \\ \hline 
  $\cf{1929}$&$92/3$&$667/45$&$2$&$192$&$57$&$84.140$&$26.730$ \\ \hline 
  $\cf{1937}$&$28$&$397/30$&$2$&$18$&$26$&$12.890$&$10.270$ \\ \hline 
  \end{tabular}\label{tab:cubic-pg-1}
\end{table}

\begin{thm}\label{thm:general-type-nonprincipal-except-229}
  Let $K$ be a cubic field of discriminant greater than $229$ and narrow
  class number $2$ such that the geometric genus of $H_K$ is at most $1$.
  Then the component of $H_K$ or $\hat H_K$
  corresponding to the nonprincipal genus is
  of general type.
\end{thm}

\begin{proof}
  Again, this is computed as above, with the same argument showing that
  only one of $H_K, \hat H_K$ need be considered.
  See Tables \ref{tab:cubic-nonprincipal-g0}, \ref{tab:cubic-nonprincipal-g1}.
  These tables were computed by choosing
  the ideal $I$ representing the nonprincipal genus to be of minimal norm
  among integral ideals that are not narrowly principal.
\end{proof}

\begin{table}[h!]
  \caption{Hilbert modular threefolds for the nonprincipal genus for
    fields $K$ with $h^+ = 2$ for which $p_g(H_K) = 0$,
    all but the first shown to be of general type.}
  \begin{tabular}{|l|c|c|c|c|c|c|}
  \hline
  $K$&$-2\zeta_K(-1)$&$c_{\O_K,\otp}$&$r$&$n$&$t$&$t'$\\
  \hline
  $\cf{229}^*$&$4/3$&$14/9$&$8$&$9$&$1.360$&$0.850$ \\ \hline 
  $\cf{257}$&$4/3$&$1$&$20$&$16$&$1.960$&$1.580$ \\ \hline 
  $\cf{697}$&$16/3$&$7/6$&$77$&$27$&$5.020$&$3.900$ \\ \hline 
  $\cf{761}$&$20/3$&$9/4$&$37$&$12$&$2.590$&$1.610$ \\ \hline 
  $\cf{788}$&$28/3$&$32/9$&$561$&$61$&$28.220$&$18.210$ \\ \hline 
  $\cf{892}$&$40/3$&$113/18$&$547$&$66$&$34.480$&$22.740$ \\ \hline 
  $\cf{985}$&$28/3$&$16/9$&$109$&$47$&$12.220$&$10.000$ \\ \hline 
  $\cf{1076}$&$44/3$&$127/30$&$281$&$67$&$25.830$&$20.500$ \\ \hline 
  $\cf{1257}$&$16$&$53/12$&$141$&$74$&$23.270$&$19.960$ \\ \hline 
  $\cf{1396}$&$64/3$&$40/9$&$1274$&$121$&$82.220$&$57.000$ \\ \hline 
  $\cf{1489}$&$16$&$4/3$&$182$&$40$&$17.560$&$12.700$ \\ \hline 
  $\cf{1593}$&$64/3$&$311/90$&$417$&$104$&$49.450$&$37.640$ \\ \hline 
  \end{tabular}\label{tab:cubic-nonprincipal-g0}
\end{table}

\begin{table}[h!]
  \caption{Hilbert modular threefolds for the nonprincipal genus for
    fields $K$ with $h^+ = 2$ for which $p_g(H_K) = 1$,
    all of general type.}
  \begin{tabular}{|l|c|c|c|c|c|c|c|}
  \hline
  $K$&$-2\zeta_K(-1)$&$c_{\O_K,\otp}$&$r$&$n$&$t$&$t'$\\
  \hline
  $\cf{1016}$&$52/3$&$73/9$&$3414$&$117$&$220.450$&$96.780$ \\ \hline 
  $\cf{1129}$&$44/3$&$173/36$&$472$&$39$&$33.160$&$18.250$ \\ \hline 
  $\cf{1229}$&$56/3$&$313/45$&$255$&$97$&$48.970$&$40.720$ \\ \hline 
  $\cf{1492}$&$68/3$&$35/9$&$850$&$146$&$90.950$&$67.550$ \\ \hline 
  $\cf{1765}$&$92/3$&$382/45$&$3649$&$100$&$310.810$&$75.830$ \\ \hline 
  $\cf{1825}$&$68/3$&$35/18$&$342$&$60$&$52.250$&$36.420$ \\ \hline 
  $\cf{1929}$&$92/3$&$53/9$&$1789$&$112$&$190.440$&$75.420$ \\ \hline 
  $\cf{1937}$&$28$&$21/4$&$134$&$101$&$73.210$&$64.580$ \\ \hline 
  \end{tabular}\label{tab:cubic-nonprincipal-g1}
\end{table}

We now discuss the three cubic fields for which we can prove that
$\kappa_{H_K} > 0$ but not that $H_K$ is of general type, aside from the field
of discriminant $697$ for which this is already known \cite{gr-dcs}.  These
correspond to the last two rows of 
Table \ref{tab:cubic-not-gt} and the tenth row of Table \ref{tab:cubic-pg-0},
where $-2\zeta_K(-1)$ is less than the scaled
volume of the union of polyhedra, so we cannot conclude
that $H_K$ is of general type.

\begin{prop}\label{prop:disc-469}
  \begin{enumerate}
  \item Let $K_1$ be the cubic field $\cf{404}=\Q(\alpha)$, where $\alpha$
    is a root of $x^3 - x^2 - 5x - 1$.  Then $\kappa_{H_K} \ge 0$.
  \item Let $K_2$ be the cubic field $\cf{469}=\Q(\alpha)$, where $\alpha$
    is a root of $x^3 - x^2 - 5x + 4$.  Then $\kappa_{H_K} > 0$.
  \item Let $K_3$ be the cubic field $\cf{788}=\Q(\alpha)$, where $\alpha$
    is a root of $x^3 - x^2 - 7x + 3$.  Then $\kappa_{H_K} > 0$.
  \end{enumerate}
\end{prop}

\begin{proof} In each case we compare the dimension of the
  space of modular forms to the defect for small $q$.  For $K_1$ we
  use \cite[Theorem 3.10]{tv-rhmf} to find that $\dim M_4 = 13$, while
  Algorithm~\ref{alg:defect} shows that the defect for forms of weight $4$ is
  $12$.  Thus $\hhh^0(K^{\otimes 2}) \ge 1$ and $\kappa_{H_{K_1}} \ge 0$.  Similarly,
  for $K_2$ the dimension and defect are $15$ and $12$ respectively. 
  Thus $\hhh^0(K^{\otimes 2}) \ge 3$ and $\kappa_{H_{K_2}} > 0$.
  Likewise for $K_3$ the values are $34$ and $27$.
\end{proof}

\begin{myremark}
  For $K_2$, we compute that $0 \le k \le 7$ the dimension of $M_{2k}$
  is $1, 1, 15, 64, 172, 365, 668, 1098$ respectively.
  The first two defects are as always $0, 1$, but we compute the next few as
  $12, 60, 170, 365, 670, 1111$, and from then on they are presumably always
  greater than the dimension of $M_{2k}$, so nothing further is learned.
  We see that $\hhh^0(K^{\otimes 3}) \ge 4$ and
  $\hhh^0(K^{\otimes 4}) \ge 2$, which would also
  suffice to prove that $\kappa_{H_K} \ge 1$.
  Similarly, for $K_3$ the dimensions of $M_{2k}$ are
  $1, 1, 34, 146, 402, 850, 1556, 2562$ and the defects are
  $0, 1, 27, 139, 391, 843, 1547, 2563$.
  In contrast, for $K_1$ the
  dimensions are $1, 1, 13, 53, 144, 304, 557, 915$, while the initial defects
  are $0, 1, 12, 59, 166, 356, 653, 1082$, and only $K^{\otimes 2}$ is seen
  in this way to have nonzero sections.  Of course, if $K^{\otimes 2}$ has nonzero
  sections then so does $K^{\otimes 2n}$ for all $n \ge 0$.
\end{myremark}

\begin{myremark}\label{rem:simple-defect}
  As alluded to in Remark~\ref{rem:answer-simple},
  it appears that the defect series $\sum_{q=0}^\infty \delta(q)t^q$ is the
  rational function
  \begin{equation*}
    \frac{-x^6 + 12x^5 + 38x^4 + 51x^3 + 36x^2 + 10x + 1}{(1-x)^2(1-x^2)(1-x^3)}.
  \end{equation*}
  This could be proved by a calculation like that used to prove
  \cite[Theorem 1]{gr-dcs}, but with much greater effort because we end up
  with a union of $9$ convex polyhedra rather than a single one
  as in \cite{gr-dcs}.
  For other cubic fields, we find a similar formula, given by a polynomial
  of degree $6$ divided by $(1-x)^2(1-x^2)(1-x^3)$.  When the inverse different
  is replaced by some other fractional ideal containing the inverse different,
  it appears that the denominator changes to $(1-x)(1-x)^2(1-x^3)(1-x^d)$,
  where $d$ divides the index.
\end{myremark}

We have a similar result for the smallest cubic field of narrow class
number $2$ and the nonprincipal genus.

\begin{prop}\label{prop:kappa-ge-0-for-229}
Let $K_3$ be the cubic field $\cf{229}=\Q(\alpha)$, where $\alpha$
    is a root of $x^3 - 4x - 1$, and let $A$ represent the nonprincipal genus.
    Then $\kappa_{H_{K;A}} \ge 0$.
\end{prop}

\begin{proof} The argument is essentially identical to the above;
  we have $\dim M_4 = 6$, while the defect for the cusp relative to
  the full group of totally positive units is $3$.  By
  Lemma~\ref{lem:defect-of-subgroup}, the defect for the group of totally
  positive units is $2(3-1)+1 = 5$, and we have $\hhh^0(K^{\otimes 2}) \ge 1$.
\end{proof}

\subsection{Fields for which the geometric genus is greater than $1$}\label{sec:pg-greater-than-1}
To close this section, we describe what would be required to prove that
$H_{K;A}$ is of general type for all $K$ with $p_g(H_K) > 1$.
We begin by showing that the criterion of \cite[Theorem 1]{tsuyumine}
applies to all genera.

\begin{thm}\label{thm:tsuyumine-all-components}
  Let $K$ be a totally real field of degree $d > 2$ with class number
  $h$, regulator $R$, and discriminant $\D_K$.  Suppose that $K$ is not
  $\Q(\zeta_n)^+$ for $n \in 7, 9, 15, 20$, and that if $d = 4$ then $3$ is not
  a square in $K$.  If
  \begin{equation*}
    2^{-2d+2} \pi^{-2d} d^d \frac{d_K \zeta_K(2)}{hR} > 1
  \end{equation*}
  then all components of $H_K$ are of general type.
\end{thm}

\begin{proof} The statement is the same as that of
  \cite[Theorem 1]{tsuyumine} except that we do not restrict to the principal
  genus and (for simplicity) do not allow the quotient by a nontrivial subgroup
  of $\Aut K$.  The proof is a straightforward adaptation of Tsuyumine's to this
  slightly more general situation.

  In our notation Tsuyumine defines 
  $\theta(I,q) = \{\nu \in I \D_K^{-1}: \nu \gg 0, \tr(\nu \beta \le q) \textup{ for some } \beta \gg 0 \in I^{-1}$ (cf.~Definition \ref{def:knoller-notation}),
  according to which this would be $\Lambda_{q+1}(I^{-1})$) and
  $u(I,q) = \#(\theta(I,q)/{\O_K^\times}^2)$.  
  Then \cite[Lemma 4]{tsuyumine} (again in our notation, and simplifying by
  ignoring the possibility of a subgroup of $\Sy_n$ acting on the coordinates)
  shows that
  $\dim S_{2k}^{(m)}(\Gamma) \ge \dim S_{2k}(\Gamma) - \sum_{\fra_\lambda} u(\fra_\lambda^2,\frac12k + m - 1)$.  Here $\fra_\lambda$ runs over a set of ideals corresponding
  to the cusps and $S^{(m)}$ denotes a certain space of cusp forms \cite[p.~271]{tsuyumine} such that if $S^{(1)} \ne 0$ and $K$ is not one of the exceptional
  fields in the statement of the theorem then the Hilbert modular variety is of
  general type.

  Let $g$ be the number of components of the Hilbert modular variety.
  Asymptotically the dimension of the space of weight-$2k$ cusp forms on each
  component is equal by Proposition~\ref{prop:asymptotic-dim-hmf},
  so on each component we obtain
  \begin{equation*}
    2^{-2+1}\,\pi^{-2d}\,D_K^{3/2}\,\zeta_K(2)\,(2k)^d + O(k^{d-1})
  \end{equation*}
  (\cite[p.~274]{tsuyumine}), while the inequality
  \cite[Lemma 5]{tsuyumine} $u(I,q) \le (2^{d-1}\,d^{-d}\,D_K^{1/2}\,|R|) q^d + O(q^{d-1})$
  is independent of $I$.  Now apply \cite[Lemma 4]{tsuyumine} with
  $\fra_\lambda$ running only over ideals representing the cusps of one
  particular component.  As in \cite[(9)]{tsuyumine} we find that
  for each component
  \begin{equation*}
    \dim S_{2k}^{(m)} \Gamma \ge \left(2^{-2d+1}\pi^{-2d}D_K^{3/2}\zeta_K(2) - 2^{-1}d^{-d}D_K^{1/2} h R\right) k^d + O(k^{d-1}),
  \end{equation*}
  whence as in \cite[Theorem 1]{tsuyumine}
  if $2^{-2d+2}\pi^{-2d}d^d \frac{D_K \zeta_K(2)}{h R} > 1$ then the
  component is of general type.  In this context Tsuyumine's $\hat h$ is
  equal to our $h$ since there is no action by a nontrivial subgroup of
  $\Sy_n$.
\end{proof}

\begin{cor}\label{grundman-all-genera}
  Let $K$ be a totally real cubic field of discriminant greater than
  $2.77 \cdot 10^8$ and $A$ an arbitrary genus of $K$.  Then $H_{K;A}$ is of
  general type.
\end{cor}

\begin{proof} In view of Theorem~\ref{thm:tsuyumine-all-components},
  the calculation of \cite[Theorem 2]{gr-cl} applies equally to all genera.
\end{proof}

Now we would like to determine all $K$ for which \cite[p. 276]{tsuyumine}
does not immediately show this to be the case.

\begin{hyp}\label{hyp:automatic-except-421}
  For all but $421$ totally real cubic fields, the largest of their
  discriminants being $26601$, the Tsuyumine-Grundman criterion
  $\frac{\disc_K \zeta_K(2)}{hR} \ge \frac{16\pi^6}{27}$
  (\cite[Theorem 1]{tsuyumine}; \cite[Corollary 10]{gr-cl}) is satisfied.
\end{hyp}

In light of Grundman's observation that this holds for all $K$ with
$\disc_K > 2.77 \cdot 10^8$, this could be proved by a finite calculation
(though one should note that the tables of cubic fields in the LMFDB
\cite{lmfdb} are not complete this far out).  If cubic fields
are ordered by discriminant, the field of discriminant $26601$ is $1133$rd,
and we have checked that there are no further counterexamples among the first
$25000$ fields.

Let us assume Hypothesis \ref{hyp:automatic-except-421}.  Of the $421$ fields,
there are $47$ with $p_g(H_K) \le 1$, and these have been studied earlier in
this section.  Of the remaining $374$, there are $25$ with class number
greater than $1$, so we need a lemma to describe the cusps.

\begin{lem}\label{lem:cusp-types}
  Let $K$ be a totally real field and fix a genus
  $A \in \Cl^+(K)/2\Cl^+(K)$.  The $h(K)$ cusps of $H_{K;A}$ are of type
  $(I^2A,\otp)$ as $I$ runs over the ideal classes of $\O_K$.
\end{lem}

\begin{proof} Immediate from \cite[Proposition 3.3.8 (a)]{hmf-mult}.
\end{proof}

Now we need only apply the existing code to the cusps as described above
and compare their defect contribution to $-2\zeta_K(-1)$ to prove that
the Hilbert modular varieties are of general type.  For example, let us
consider the first cubic field of class number $2$, which does not
satisfy the Tsuyumine-Grundman criterion.

%%% make sure that I am multiplying by a correct normalization factor
\begin{ex}\label{ex:disc-1957}
  The field $K = \cf{1957}$ is $\Q(\alpha)$, where $\alpha$ is a root of
  $x^3 - x^2 - 9x + 10$.
  We have $h_K = 2$, but the narrow class group
  is cyclic of order $4$, so there are $2$ genera.  Let the narrow class
  group be represented by $I_0, I_1, I_2, I_3$, where $I_0, I_2$ are
  principal and $I_0$ is narrowly principal.
  We may choose the representatives to have norm $1, 2, 4, 4$ respectively
  (this specifies them uniquely).  For the principal
  genus and $\PGL_2^+$, the two cusps are of type
  $(I_0,\otp), (I_2,\otp)$, and for the
  nonprincipal genus, they are of type $(I_1,\otp), (I_3,\otp)$.
  If we used $\PSL_2$ instead of $\PGL_2^+$, the groups of units would be
  ${\O^\times_K}^2$ instead of $\otp$.

  We calculate that $-2\zeta_K(-1) = 104/3$, and that the volume of the
  unions of polyhedra for cusps of type $(I_j,{\O^\times_K}^2)$ scaled by the
  covolume of ${\hat I}_j$ are respectively
  $34/3, 71/18, 44/15, 103/18$.  Since $104/3 > 34/3 + 44/15, 71/18+103/18$,
  this shows that both $H_{K;A}$ are of general type.
\end{ex}

\begin{hyp}\label{hyp:all-general-type}
  For all $K$ with $p_g(H_K) > 1$ and all genus representatives $A$,
  the Hilbert modular threefold $H_{K;A}$ is of general type.
\end{hyp}

\begin{myremark}\label{rem:pg-gt-1-not-gen-type-quad}
  The analogous statement is false for quadratic fields;
  in the honestly elliptic case of \cite[Theorem VII.3.3]{vdg} there are
  many counterexamples.
\end{myremark}

In \cite{code} there is a script that is expected to verify
that Hypothesis~\ref{hyp:automatic-except-421} implies
Hypothesis~\ref{hyp:all-general-type}
if run long enough, as well as code that will check that there are no
counterexamples to Hypothesis~\ref{hyp:automatic-except-421}
among cubic fields in the LMFDB.

\section{Results at higher level}\label{sec:results-higher-level}
In this section we describe our results on the Kodaira dimension
of Hilbert modular threefolds of the form $H_{K,I;A}$.  As before,
we can try to prove that such a surface is of general type by showing that
the dimension of the space of modular forms grows faster than the defects,
and if this fails we can still hope to prove that $\kappa_{H_{K,I;A}} > 0$ by
finding $q$ such that $\dim \left| qK \right| \ge 2$.  If $p_g(H_K) = 1$ then
we have already shown that $H_K$ is of general type, so the same follows for
its covers and so the most interesting cases are those with $p_g(H_K) = 0$.

We now begin to list the pairs $K, I$ systematically for which the geometric
genus of $H_{K,I}$ is at most $1$.  First note that if $J|I$ and
$p_g(H_{K,J}) > 0$, then by the theory of oldforms we have
$p_g(H_{K,I}) > 2p_g(H_{K,J}) \ge 2$, so such cases can be ignored.
%% (However, some should still be considered when the narrow class number is
%% greater than $1$, because there would still be interesting components with
%% $p_g \le 1$.)
We will use the trace formula \cite[(5.1.2)]{hmf-mult}
to bound the levels for which the geometric genus is at most $1$.
Let us begin by stating the formula, first introducing the notation of
\cite[(5.1.3), (5.2.2)]{hmf-mult}.
\begin{notation}\label{not:trace-formula}
  Let $K$ be a totally real field of degree $d$ and let $I$ be an ideal
  of its maximal order.  Let $\Sy$ be the
  set of orders containing one of the $\O_K[x]/(x^2-tx+u)$ where
  $u$ ranges over totally positive units modulo squares and $t$ over
  elements of $\O_K$ such that $t^2 - 4u \gg 0$.
  For each order
  $S \in \Sy$, let $c_S = h(S)/2[S^\times:\O_K^\times]$ as in
  \cite[(5.1.3)]{hmf-mult}.  Let $A = (-1)^{d-1} h^+(\O_K)$, let
  $B = \frac{1}{2^{d-1}} \cdot \left|\zeta_K(-1)\right| \cdot N(I) \prod_{\p|I} \left(1+N(\p)^{-1}\right)$,
  let $C(u,t) = \frac12 \sum_{S \in \Sy} \frac{h(S)}{[S^\times:\O_K^\times]} m(\hat S,\hat \O_K; \hat \O_K^\times)$,
  where the $m$ are certain embedding numbers that are products over primes
  dividing the level, defined in \cite[Section 30.6]{voight},
  and let $D_k(u,t) = \frac{\alpha(u,t)^{k+1}-\beta(u,t)^{k+1}}{\alpha(u,t)-\beta(u,t)}$,
  where $\alpha(u,t), \beta(u,t)$ are the roots of $T^2 - tT + u$.
\end{notation}

\begin{thm}\cite[(5.1.2)]{hmf-mult} \label{thm:dim-formula}
  The dimension of the space of Hilbert cusp forms for $\Gamma_0(I)$ of
  weight $2k$ is the coefficient of $T^{2k}$ in
  \begin{equation*}
    AT^2 + BT\left(T \frac{d}{dT}\right)^n \left(\frac{T}{1-T^2}\right) +
  (-1)^n \sum_{(u,t)} C(u,t)\sum_{m \ge 1} N_{K/\Q}(D_{2m-2}(u,t)) T^{2m}.
  \end{equation*}
\end{thm}

\begin{myremark}\label{rem:hmf-mult-error}
  We have corrected a typographical error in \cite[(5.1.2)]{hmf-mult},
  where the sign before the term beginning with $\sum_{(u,t)} C(u,t)$ is
  omitted.
\end{myremark}

\begin{thm}\label{thm:prime-bound}
  Let $K$ be a totally real cubic field satisfying $p_g(H_K) = 0$
  and let $\p$ be a prime ideal of $\O_K$.      
  If $p_g(H_{K,\p}) \le 1$ then either $\p$ divides the conductor of
  some $S \in \Sy$ or
  $h^+(\O_K) + (N(\p)+1)\left|\zeta_K(-1)\right|/4 - 2\sum_{S \in \Sy} c_S \le 1$.
\end{thm}

\begin{myremark} Of course this implies an effective upper bound for $N(\p)$.
\end{myremark}

\begin{proof} This is a consequence of \cite[(5.1.2),(5.1.3)]{hmf-mult}.
  Indeed, the geometric genus is the dimension of the space of weight-$2$
  cusp forms by Theorem~\ref{thm:can-div-hmv}.  We estimate the constants
  $A, B, C, D$ in this case.  The coefficient of $T^2$ in
  $T\left(T\frac{d}{dT}\right)^n \left(\frac{T}{1-T^2}\right)$ is always
  $1$, and $(N(\p)+1)\left|\zeta_K(-1)\right|/4$ is the value of $B$ in
  this context (recall that $p_g = 0$ implies $h_K = 1$ for cubic fields).
  The coefficient $D_0$ is always $1$ so that factor may be ignored in
  computing the dimension of the space of weight-$2$ cusp forms, while
  \cite[Lemma 30.6.17]{voight} implies that the local embedding number
  $m(\hat S,\hat \O; \hat \O^\times)$ is at most $2$ when the level is an
  unramified prime.  The result follows.
\end{proof}

We now present the calculation of levels at which the geometric genus is
$0$ or $1$ in detail for the smallest cubic field.
\begin{thm}\label{thm:pg-01-disc-49}
  Let $K = \qz7$.  The Hilbert modular variety $H_{K,I}$ is
  of geometric genus $0$ if and only if
  $I \in \{(1), \p_7, (2), \p_{13}, \p_{29}, \p_{43}\}$, and $1$ if and only if
  $I \in \{(3), \p_{41}, \p_7^2, 2\p_7, (4), \p_{71}, \p_{7}\p_{13}, \p_{97}, \p_{113}, \p_{127}, \p_{13}^2\}$.
\end{thm}

\begin{proof} We apply Theorem~\ref{thm:prime-bound}.
  We have $h^+(\O_K) = 1$, while
  $\zeta_K(-1) = -1/21$ and
  $\sum_{S \in \Sy} c_s \le 1/2(1/2 + 2(1/3) + 6(1/7)) = 85/84$, where each term
  $1/2, 2(1/3), 6(1/7)$ appears in the sum if and only if $\p$ splits in
  $K(i), K(\zeta_3), K(\zeta_7)$ respectively (we assume that $\p$ is not one
  of the primes $\p_7, (2), (3)$ dividing one of the conductors).
  It follows that if
  $N(\p) > 169$ then $p_g(H_{K,\p}) > 1$, and one checks cases up to there.

  Next we allow the level to be the product of two primes.  If the primes
  are distinct then the local embedding number is
  $\prod_{i=1}^2 1+\left(\frac{N(\p_i)}{d}\right)$, and so the contribution cannot
  exceed $2(85/42)$.  So again if $N(\p_1 \p_2) > 340$ we have $p_g > 1$,
  and it is not difficult to check up to this point.  At level $\p^2$ the
  local embedding numbers are the same as for $p$ except for primes above
  $2, 3, 7$, and again it is straightforward to verify that $p_g > 1$ for
  primes of norm $27$ or greater.

  Finally, if the level is a product of three or more primes, then
  $p_g$ can be less than $2$ only if $p_g = 0$ for all
  proper divisors.  However, it has already been verified that if
  $p_g = 0$ then $I$ is prime, so this is not possible.
\end{proof}  

We now study the Kodaira dimension of these covers.

\begin{defn}\label{def:ell-defect}
  Let $\p$ be a prime of $\O_K$, where $K$ is a totally real cubic field.
  If $K = \qz7$, then define $c_\p$ to be $2, 1, 0$ according as
  $N(\p) \equiv 1, 0, -1 \bmod 7$.  If $K = \qz9$, then let
  $c_\p = 2, 1, 0$ according as $N(\p) \equiv 1, 0, -1 \bmod 3$.
  For a general ideal $I$ of $\qz7$ or $\qz9$, let $e_I = 0$ if $I$ is
  divisible by the square of the ramified prime and otherwise
  $\frac{\prod_{\p|I} c_\p}{c_K}$, where $c_K = 84$ for $\qz7$ and $18$ for
  $\qz9$.  For all other fields define $c_\p = e_I = 0$.
\end{defn}

\begin{thm}\label{thm:kod-dim-prime-level}
  Let $K$ be a totally real cubic field with $h(K) = 1$, let
  $z = -2\zeta_K(-1)$, let $A$ be a genus representative for $K$,
  and let $v$ be the volume of the cusp for $A$
  as computed in Algorithm~\ref{alg:volume}.
  Let $\p$ be a prime ideal of $\O_K$ such that
  $(N(\p)+1) z > 2v + e_\p$.  Then $H_{K,\p;A}$ is of general type.
\end{thm}

\begin{proof} Just as with the modular curve $X_0(p)$, the Hilbert modular
  variety $H_{K,\p;A}$ has two cusp orbits when $p$ is prime, corresponding
  to $\infty, 0$.
  The stabilizers are respectively the groups of upper
  and lower triangular matrices contained in $\Gamma_0(\p;A)$; in particular,
  if we conjugate in $\GL_2(K)$ to an upper triangular subgroup of $\SL_2$,
  the group of units in the stabilizer is the full group of squares of units.
  Since the class number is $1$, it follows that the defects of both
  cusps are equal to the defects of the cusp for $H_{K;A}$.
  In view of Example~\ref{ex:cq-7-9}, the leading coefficient
  in the asymptotic to the contribution of elliptic points to the defect of
  $H_{K,\p}$ is given by $e_\p$, since the number of quotient singularities
  is multiplied by $1+\left(\frac{\p}{K(\alpha)}\right)$, where $K(\alpha)$
  is the quadratic field corresponding to the elliptic point.
  On the other hand,
  the index of $\Gamma_0(\p)$ in $\SL_2(\O_K)$ is $N(\p)+1$, and 
  the asymptotic dimension of the space of modular forms
  of weight $k$ is multiplied by this factor.
  %% We only need the leading term of the dimension of the space of modular
  %% forms of weight $2n$, which is asymptotically
  %% $2\left|\zeta_K(-1)\right|[\SL_2(\O_K):\Gamma_0(\p)] n^3 = 2(N(\p)+1)zn^3$
  %% by \cite[Theorem 11]{shimizu}, and of the defects, which
  %% by Algorithm~\ref{alg:volume} we find to be $5n^3/12$ for each cusp.
  %% For $N(\p)+1 > 8$ the first of these exceeds the second and we conclude that
  %% $H^0(nK) > cn^3$ for $c = 4(N(\p)+1)/21 - 2(5/12)$.  The result follows.
\end{proof}

\begin{myremark} The classification of cusps given here is a special case of
  the results of \cite[Sections 3.1--3.3]{hmf-mult}.
\end{myremark}

\begin{notation}
  We use $\p_n$ to denote a prime of norm $n$.  If $q$ is a rational prime that
  factors as $\frP_1\frP_2^2$ in $\O_K$, we denote $\frP_1, \frP_2$ by
  $\p_q, \rr_q$ respectively.
\end{notation}

It happens not to be necessary to distinguish the factors of primes that
split completely.  In the case of a Galois extension it is understood that
the Galois images of all ideals mentioned have the same property.
In view of Theorem~\ref{thm:pg-same} in the Appendix it is unnecessary to
distinguish between different genera, and so the rows referring to
$\cf{229}, \cf{257}$ describe both genera.

\begin{table}[h!]
  \caption{Fields and primes for which
  $p_g(H_{K,\p}) \le 1$.}
  \begin{tabular}{|l|c|c|}
    \hline
    $K$&$\p: p_g = 0$&$\p: p_g = 1$\\ \hline
    $\cf{49}$&$\p_7, (2), \p_{13}, \p_{29}, \p_{43}$&$(3), \p_{41}, \p_{71}, 
    \p_{97}, \p_{113}, \p_{127}$\\ \hline
    $\cf{81}$&$\p_3, \p_{19}, \p_{37}$&$(2), 
    \p_{17}, \p_{73}$\\ \hline
    $\cf{148}$&$\p_2, \p_5, \p_{13}$&$\p_{17}, 
    \p_{25}$\\ \hline
    $\cf{169}$&$\p_5, \p_{13}$&$(2)$\\ \hline
    $\cf{229}$&$\p_2, 
    \p_4, \p_7$&$$\\ \hline
    $\cf{257}$&$\p_3, \p_5, \p_7$&$$\\ \hline
    $\cf{316}$&$\rr_2$&$\p_2$\\ \hline
    $\cf{321}$&$\rr_3$&$\p_3, \p_7$\\ \hline
    $\cf{361}$&$$&$\p_7$\\ \hline
    $\cf{404}$&$$&$\p_2$\\ \hline
    $\cf{469}$&$$&$\p_4$\\ \hline
    $\cf{568}$&$$&$\rr_2$\\ \hline
  \end{tabular}\label{table:pg-01-prime}
\end{table}

\begin{prop}\label{cor:pg-01-prime-table}
  Table \ref{table:pg-01-prime} gives the pairs consisting of a field and
  a prime $\p$ for which $p_g(H_{K,\p}) \le 1$.
\end{prop}

\begin{proof}
  See \cite{code}.  Our implementation is not very efficient, but it is
  fast for all but the two smallest real cubic fields, so we did not
  feel a need to improve it.
\end{proof}

\begin{prop}\label{prop:mostly-gt}
  The Hilbert modular varieties $H_{K,\p}$ for the primes in
  Table~\ref{table:pg-01-prime} are of general
  type, except possibly as shown in Table~\ref{table:pg-01-ngt}.
\end{prop}

%% In the most delicate cases, we have cusp volume 5/12, 3/4 for \qz7, \qz9.
%% So, for example, (2) in \qz7 gives e_p = 1/42 and we find that
%% 9*2/21 = 2*5/12 + 1/42.
\begin{table}[h!]
  \caption{Fields, genera, and primes for which we cannot show that
    $H_{K,\p,I}$ is of general type.  The $g$ column is $+$ for the narrowly
    principal genus when $h^+ = 2$; otherwise it is left blank.  There
    are no examples with the nonprincipal genus.}
    \begin{tabular}{|l|c|c|}
      \hline
      $K$&$g$&$\p$\\ \hline
      $\cf{49}$&$$&$\p_7, (2)$\\ \hline
      $\cf{81}$&$$&$\p_3$\\ \hline
      $\cf{148}$&$$&$\p_2$\\ \hline
      $\cf{229}$&$+$&$\p_2$\\ \hline
      $\cf{257}$&$+$&$\p_3$\\ \hline
    \end{tabular}\label{table:pg-01-ngt}
\end{table}

\begin{proof}
  Use \cite{code} to check the condition of
  Theorem~\ref{thm:kod-dim-prime-level}.
\end{proof}

\begin{table}[h!]
  \caption{Fields and nonprime ideals which
  $p_g(H_{K,I}) \le 1$.} 
  \begin{tabular}{|l|c|c|}
    \hline
    $K$&$I: p_g = 0$&$I: p_g = 1$\\ \hline
    $\cf{49}$&$$&$\p_7^2, 2\p_7, (4), \p_7\p_{13}, \p_{13}\p_{13}'$\\ \hline
    $\cf{81}$&$\p_3^2$&$(3), \p_3\p_{19}$\\ \hline
    $\cf{148}$&$\p_2^2, \p_2\p_5$&$(2), \p_2^2\p_5$\\ \hline
    $\cf{169}$&$$&$\p_5\p_5'$\\ \hline
    $\cf{229}$&$\p_2^2, \p_2^3$&$$\\ \hline
    $\cf{316}$&$\rr_2^2$&$\rr_2^3$\\ \hline
    $\cf{321}$&$\rr_3^2$&$$\\ \hline
  \end{tabular}\label{table:pg-01}
\end{table}

\begin{prop}\label{cor:pg-01-table}
  Table \ref{table:pg-01} gives the pairs consisting of a field and a nonprime
  ideal $I$ for which $p_g(H_{K,I;A}) \le 1$.
\end{prop}

\begin{proof} (Sketch.)  Such an ideal can only be divisible by primes listed
  in Table~\ref{table:pg-01-prime} and all of its divisors other than $(1)$
  must be in one of the two tables.  We verify using our implementation of
  \cite[(5.1.2),(5.1.3)]{hmf-mult} that if $I$ is in Table~\ref{table:pg-01}
  and $\p$ is in Table~\ref{table:pg-01-prime} but $I\p$ is not in
  Table~\ref{table:pg-01} then $p_g(H_{K,I}) > 1$.  As in
  Corollary~\ref{cor:pg-01-prime-table}, it is unnecessary to distinguish
  between different genera here.  Again, see \cite{code}.
\end{proof}
%%% note to self: can we do this in a way that would be more efficient for
%%% larger fields?  I do like this dependence on division, though.

\begin{prop}\label{cor:mostly-gt-not-prime}
  Let $K$ be a cubic field for which $p_g(H_K) \le 1$.
  For all proper nonprime ideals $I \subset \O_K$ and genus representatives
  $A$, the corresponding Hilbert modular variety is of general type,
  except possibly for $\p_3^2$ in $\qz9$,
  $\p_2^2$ in the field of discriminant $148$,
  and $\p_2^2$ for the field of discriminant $229$ (principal genus).
\end{prop}

\begin{proof} If $I$ has a prime divisor $\p$ for which $H_{K,\p}$ is
  of general type, then $H_{K,I}$ too is of general type, so we need only
  consider powers of the ideals from Proposition~\ref{prop:mostly-gt}
  and $2\p_7$ in $\qz7$.
  We start with $I = \p_7^2$ for $\qz7$.  The index of
  $\Gamma_0(I)$ is $N(\p_7)^2 + N(\p_7) = 56$,
  just as for the index of $\Gamma_0(p^2)$ in $\SL_2(\Z)$.  It is easily
  seen that the ray class group mod $\p_7\infty_1\infty_2\infty_3$ is of
  order $2$, and by \cite[Corollary~3.1.18]{hmf-mult} this implies that there
  are $4$ cusps.  These are represented by $\infty, 0, 1/\pi_7, 3/\pi_7$,
  where $(\pi_7) = \p_7$.  For the first two of these the group of units
  is the full group of totally positive units as before, but for the other
  two it is the subgroup of totally positive units congruent to $1 \bmod \p_7$
  and the defects are multiplied by the index of this subgroup, which is $3$.
  Thus the dimension of the space of
  modular forms of weight $2n$ is asymptotically
  $56$ times that for level $1$, which makes
  it $16n^3/3$, while the defects are asymptotic to $(8 \cdot 5/12)n^3$,
  which is smaller.  There are no elliptic points of order $7$.
  Similarly, taking $I = (2)^2$ for $\qz7$, the index is $72$, the ray class
  group mod $2\infty_1\infty_2\infty_3$ is trivial, and so there are $3$
  cusps, represented by $\infty, 0, 1/2$.  Again, the groups of units for
  the first two are ${\O_K^\times}^2$, but for the last we get the totally
  positive units congruent to $1 \bmod 2$, a subgroup of index $7$.
  In this case we have $8$ elliptic points of order $7$, of which $6$
  contribute to the defect.  So the dimension is asymptotically
  $2 \cdot 72n^3/21$, while the defect is asymptotically
  $(9\cdot 5/12 + 6/252)n^3$ and we have general type.  We leave it to the
  reader to check that $H_{\qz7,2\pi_7}$ is of general type, the dimension of
  the space of cusp forms of weight $2n$ being asymptotic to 
  $48n^3/7$ and the defect to $71n^3/42$.  (The gap is much larger in this
  case because there are only $4$ cusps, each stabilized by the full group
  of units.)  

  For $(3)$ in $\qz9$, the index is $36$ and there are $6$ cusps.
  The stabilizers of these have index $27, 3, 3, 1, 1, 1$ in the
  $\SL_2$-stabilizer, and the unit indices are respectively $1, 3, 3, 1, 1, 1$
  (the $3$ arises as in \cite[Proposition 3.3.8 (a)]{hmf} because the
  totally positive units congruent to $1 \bmod \p^2$ are of index $3$).
  Since $\zeta_{\qz9}(-1) = -1/9$,
  the dimension of the space of forms is asymptotic to
  $(36 \cdot 2/9)n^3$, while the defect is $10$ times that for level
  $1$ since there are no elliptic points of order $9$.
  We compute that the defect for the cusp at level $1$ is asymptotic
  to $3n^3/4$.  Since $36 \cdot 2/9 > 10 \cdot 3/4$, the result follows.

  The arguments for $\pi_2^3$ in \cf{148}, \cf{229}
  are similar and we will only sketch them; likewise for $\pi_3^2$ in 
  \cf{257}.  The index of $\pi_2^3$ in both cases is $12$
  and there are $4$ cusps.  For \cf{148}, all totally
  positive units are $1 \bmod \pi_2^2$ and so the unit indices are
  $1, 2, 1, 1$; thus the cusps create a defect asymptotic to
  $5\cdot 55/36 n^3$.  On the other hand, the dimension of the space of
  modular forms grows like $12(-2\zeta_K(-1)) n^3 = 8n^3$, which is larger.

  In \cf{229}, not all totally positive units are
  $1 \bmod \p_2^2$ and the unit indices are all $1$, so the dimension grows
  like $12(4/3) n^3$ and the defect like $4(3) n^3$.  Again, in $\cf{257}$,
  the index of $\Gamma_0(\p_3^2)$ is $12$, not all totally positive units are
  $1 \bmod \p_3$, and there are $3$ cusps with unit indices $1, 2, 1$.
  So the dimension of the space of modular forms and the defect grow like
  $12(4/3) n^3$ and $4(26/9) n^3$ respectively.
\end{proof}

\begin{myremark} The argument does not apply to $\p_3^2$ in $\qz9$.
  The index is $12$ and the ray class group has order $2$, so again there are
  $4$ cusps.  Again the intermediate cusps $1/\pi_3, 2/\pi_3$ are
  stabilized by totally positive units congruent to $1 \bmod \p_3$, but this
  time that is all of them, so the defects are the same for all cusps.
  The dimension of the space of modular forms of weight $2n$ is $12$ times
  that for level $1$, so $8n^3/3$, while the defect is $4$ times that for
  level $1$ and is asymptotic to $3n^3$.  Similarly for $\p_2^2$ in \cf{148};
  the dimension of the space of modular forms
  is asymptotic to $4n^3$ and the defect sum to $55n^3/12$.  Again,
  for $\p_2^2$ in \cf{229}, the asymptotics are
  $8n^3$ and $9n^3$ respectively, and we do not have a general type result.
\end{myremark}  

\begin{myremark} For reasons discussed at the end of
  Section~\ref{sec:results-level-1}, we certainly expect that the hypothesis
  that $p_g(H_K) < 1$ is unnecessary.
\end{myremark}

\begin{prop}\label{prop:pos-kod-dim}
  In Table~\ref{tab:positive-kod}, the varieties
  $H_{K,I}$ have at least the indicated Kodaira dimension.
\end{prop}

\begin{proof} Table~\ref{tab:positive-kod} also gives an $n$ such that
  the dimension of the space of Hilbert cusp forms of level $I$ and weight
  $n$ is equal to resp.~greater than the sum of the defects
  for $\kappa = 0, 1$ respectively.  See \cite{code} for details.
  %% For $K = \qz7, I = (2)$, the only case for which the two asymptotics are
  %% equal, we compute the two series explicitly, finding that the Hilbert
  %% series for the ring of modular forms, the defect series for each of the
  %% $2$ cusps, and the defect series for the $6$ noncanonical elliptic points
  %% of order $7$ are (fill in)
  %% \begin{align*}
  %%   f_1 = \frac{(l^{12} - 3l^{11} + 4l^{10} + 10l^9 + 8l^8 + 10l^7 + 12l^6 + 11l^5 + 9l^4 + 
  %% 8l^3 + 2l^2 - 2l + 2)/(l^{11} - 2l^{10} + 2l^8 - l^7 - l^4 + 2l^3 - 2l + 1)},\\
  %%   &\quad f_2 = l\frac{1+3l^2+5l^3+5l^4+2l^5-l^6}{(1-l)^2(1-l^2)(1-l^3)},f_3 = \frac{l^6}{(1-l)(1-l^2)(1-l^3)(1-l^7)}
  %% \end{align*}
  %% Thus the generating function for the dimension less the sum of the defects is
  %% $f_1-2f_2-6f_3$.  I thought this would lead to k = 2, but I was wrong.
\end{proof}

%% \begin{myremark} We believe that Table~\ref{tab:positive-kod}
%%   should include the cases of
%%   level $\p_2^2$ in \cf{229} as well as level $\p_3$ in \cf{257}.
%%   However, in the absence of a dimension
%%   formula for the individual components, we cannot yet prove it.
%% \end{myremark}
%% we can now.

\begin{table}[h!]
  \caption{Fields and levels for which the Hilbert modular variety can be
    shown to be of nonnegative Kodaira dimension, but not of general type.}
  \begin{tabular}{|l|c|c|c|c|c|c|}
    \hline
    $K$&$I$&type&$\kappa$&$n$&$\dim M_n$&$\sum_i \delta_i(n)$ \\ \hline
    $\cf{49}$&$(2)$&$ $&$1$&$4$&$6$&$4$\\ \hline
    $\cf{81}$&$\p_3^2$&$ $&$0$&$4$&$13$&$12$\\ \hline
    $\cf{148}$&$\p_2^2$&$ $&$1$&$4$&$18$&$15$\\ \hline
    $\cf{229}$&$\p_2^2$&$\GL_2^+$&$1$&$4$&$33$&$27$\\ \hline
    $\cf{229}$&$\p_2^2$&$\SL_2$&$1$&$4$&$30$&$27$\\ \hline
    $\cf{257}$&$\p_3$&$\GL_2^+$&$1$&$4$&$22$&$18$\\ \hline
    $\cf{257}$&$\p_3$&$\SL_2$&$1$&$4$&$20$&$18$\\ \hline 
  \end{tabular}\label{tab:positive-kod}
\end{table}

\section{Future work}
It is also natural to ask about the opposite direction: namely, can it be
proved that some of the Hilbert modular varieties for small fields and levels
are not of general type?  Unfortunately the classical methods of Hirzebruch,
van de Ven, and Zagier \cite[Chapter VII]{vdg} do not apply in dimension
greater than $2$ because of the unavailability of Hirzebruch-Zagier cycles
and lack of uniqueness of minimal models.
To the author's knowledge there is only one
such result.

\begin{thm}\label{thm:first-two-unirational}
  [Elkies-Harris, unpublished] Let $K = \qz7$ or
  $\qz9$.  Then $H_K$ is unirational.
\end{thm}

The proof relates $H_K$ to a moduli space of curves of genus $2$ with a point
of order $7$ or $9$ on the Jacobian and uses an explicit construction to prove
that this moduli space is rational.  This approach cannot be applied
to any nonabelian extension.

One might expect that there is no particular reason for the subspaces
$M_{2q}$ and $S_{2q,U_C}$ not to be in general position inside
$M_{2q,U_C}$.  If so, then the Kodaira dimension would be $-\infty$ in all
cases where the $q\/$th defect is greater than or equal to the dimension of
the space of modular forms of weight $2q$.  However, this expectation is
violated for certain Hilbert modular surfaces, so we do not believe that it
will hold for threefolds either.  We give an example.

\begin{ex}
  Let $K = \Q(\sqrt{53})$.  There are no reducers and $h^+_K = 1$.
  The dimension of the space of cusp forms of weight
  $2k$ is $7k^2/3+O(n)$,
  since $\zeta_K(-1) = 7/6$.  On the
  other hand, the area of the polygon $T_{V}(1)$ is $7/2$, so the $k$th defect
  is asymptotic to $7k^2/2$ (this can be verified from
  \cite[Proposition III.3.6]{vdg} as well; by \cite[Example II.5.1]{vdg}
  or \cite{hmf} the
  cusp components have self-intersection $-9,-2,-2,-2,-2,-2,-2$).
  The $5$ elliptic points of type $(3;1,1)$ each contribute $k^2/6$
  to the defect.  From \cite[Theorem VII.3.3]{vdg}
  the Hilbert modular surface is of Kodaira dimension $1$, meaning that the
  $n$th plurigenus is asymptotically linear in $n$.  Thus the defect
  conditions fail strikingly to be independent.  Similar examples could be
  given for surfaces of general type; with $K = \Q(\sqrt{89})$, for example,
  the Hilbert modular surface is of general type and $2\zeta_K(-1) = 26/3$
  while the cusp defect and sum of the elliptic defects are asymptotically
  $21n^2/2$ and $n^2/6$ respectively.
\end{ex}

It therefore seems fruitful to examine the $q$-expansions of Hilbert modular
forms for cubic fields of small discriminant not covered by
Theorem~\ref{thm:general-type-from-473} in order to obtain lower bounds on
the Kodaira dimension of $H_K$.  Such $q$-expansions can be computed, though
at present this cannot be done quickly.  In future work with Assaf, Costa,
and Schiavone we will describe techniques to compute the $q$-expansions and
their implications for the Kodaira dimension of Hilbert modular threefolds
for cubic fields of small discriminant.  

\appendix
\section{Dimensions of spaces of Hilbert modular forms over fields of odd degree (by Adam Logan and John Voight)}\label{app:same-dimension}

In this appendix, we show that for fields of odd degree, the dimension of the space of Hilbert modular forms supported on a connected component is equal across all components.  

As we will be adapting results from \cite{hmf-mult}, we compare the notation
from that paper to ours.  In \cite{hmf-mult}, the totally real field,
its maximal order, and its degree are $F, R, n$ rather than $K, \O_K, d$,
and the congruence subgroups that we have denoted
$\Gamma_0(I;A), \hat \Gamma_0(I;A)$ are written as
$\Gamma_\mathfrak{b}=\Gamma_0(\mathfrak{N})_\mathfrak{b},\Gamma_0^1(\mathfrak{N})_\mathfrak{b}$,
where $\mathfrak{b}$ runs over the representatives of the narrow class group
indexing the connected components of the associated Hilbert modular variety.

\begin{thm}\label{thm:pg-same}
Suppose that $d = [F:\mathbb{Q}]$ is odd.  Then for all $k=(k_i)_{i=1}^d \in 2\mathbb{Z}_{\geq 0}$, the dimensions
\[ \dim S_k(\Gamma(I;A)), \quad \dim S_k(\hat \Gamma(I;A)) \]
are independent of $A$.  
\end{thm}

\begin{proof}
For the main part of the argument, we use a dimension formula due to
Shimizu \cite[Theorem 11]{shimizu} which first requires us to suppose
that $k_i>2$ for some $i$.  For reasons of space, we do not repeat the
formula here, but we claim that it depends only on the rotation factors
of the elliptic points for $\Gamma_0(I;A)$ or $\hat \Gamma_0(I;A)$.
Indeed, the leading
term (a volume) depends only on the index $[\PSL_2(\O_K):P\Gamma_0(I;A)]$
or $[\PSL_2(\O_K):P\hat\Gamma_0(I;A)]$,
the field $K$ and the weight $k$, and the
final term (cuspidal contribution) is zero as pointed out by Shimizu
\cite[(39)]{shimizu} using that $d$ is odd.  The key term is the
middle term (elliptic points contribution): in Shimizu's notation
\[ 
  \sum_{i=1}^s \frac{1}{[\Gamma_{z_i}:1]} \sum_{\gamma \in \Gamma_{z_i},\gamma \ne 1} \prod_{i=1}^n \frac{{\alpha^{(i)}}^{r_i}}{1-\alpha^{(i)}}
\]
where $z_i$ are a complete set of inequivalent elliptic points with stabilizers $\Gamma_{z_i}$, the $\alpha^{(i)}$ are the rotation factors as in \cite[\S4.3]{hmf-mult}, and $r_i=k_i/2$.  The number of elliptic points of each order
  is the same for the different genera by \cite[Proposition 4.2.3]{hmf-mult},
  which gives a formula independent of $A$.
  Thus it suffices to prove that the rotation
  factors are the same, following \cite[\S 4]{hmf-mult}.  
  %  (which in the notation of the main body of paper, would be for $H_K$ and $H_{K;A}$ (see Notation~\ref{not:hmv}).

For points of order $q=2$, there is nothing to prove, so suppose
$q>2$.  From \cite[Lemma 4.1.1]{hmf-mult}, elliptic points of order
$q>2$ arise from an extension $L=K(\zeta_{2q}) \supset K$ of degree
$2$.  Thus $K \supseteq K^+ = \Q(\zeta_{2q})^+$ and
$[K^+:\Q]=\phi(2q)/2$ divides $d=[K:\Q]$, which is odd.
So $q=p^r$ where
$p$ is odd (and in fact, $p \equiv 3 \pmod{4}$).  Again since $n$ is
odd, there is a prime $\mathfrak{p}$ of $R$ above $p$ with odd
ramification index.  But $p$ is totally ramified in
$\Q(\zeta_{2p^r})$ with even ramification index, so
$\mathfrak{p}$ must ramify in $K$.
  
  Thus the oriented optimal selectivity condition
  \cite[\S4.3, (OOS)]{hmf-mult} fails.  By
  \cite[Definition 4.3.8, Theorem 4.3.11 (a)]{hmf-mult}, all stabilizer orders $S=R[\gamma_i]$ are orientedly genial.  In particular, the sets of rotation factors are the same across components (see the discussion following \cite[Definition 4.3.8]{hmf-mult}) and by
  \cite[Theorem 4.3.11 (c)]{hmf-mult} the multiplicities are the same as well.

We now treat the remaining cases.  If $k_i=0$ for some $i$, then either $k=0$ and $\dim S_k(\Gamma_0(I;A))=1$ (constant functions) independent of $A$; otherwise $\dim S_k(\Gamma_0(I;A))=0$ by van der Geer \cite[Lemma I.6.3]{vdg}.
Finally, we consider the case where $k_i=2$ for all $i$, 
  the most important one for us.  We just showed that the polynomials $P$ such that $P(k)$ is the
  dimension of the space of forms of even parallel weight $k$ for $k > 2$ (see
  \cite[\S1]{freitag-dim}) are the same for all genera.
  The result then follows from \cite[Satz 7.2]{freitag-dim} and the comments after it.
\end{proof}

\bibliographystyle{amsplain}

\begin{thebibliography}{10}

\bibitem{hmf}
Eran Assaf, Angelica Babei, Ben Breen, Sara Chari, Edgar Costa, Juanita
  Duque-{R}osero, Aleksander Horawa, Jean Kieffer, Avinash Kulkarni, Grant
  Molnar, Abhijit Mudigonda, Michael Musty, Sam Schiavone, Shikhin Sethi,
  Samuel Tripp, and John Voight.
\newblock Software for computing with {H}ilbert modular forms, 2024.
\newblock \url{https://github.com/edgarcosta/hilbertmodularforms} (retrieved
  June 12, 2025).

\bibitem{hmf-mult}
Eran Assaf, Angelica Babei, Ben Breen, Edgar Costa, Juanita Duque-Rosero,
  Aleksander Horawa, Jean Kieffer, Avinash Kulkarni, Grant Molnar, Sam
  Schiavone, and John Voight.
\newblock A database of basic numerical invariants of {Hilbert} modular
  surfaces.
\newblock In {\em LuCaNT: LMFDB, computation, and number theory. Conference,
  Institute for Computational and Experimental Research in Mathematics (ICERM),
  Providence, Rhode Island, USA, July 10--14, 2023}, pages 285--312.
  Providence, RI: American Mathematical Society (AMS), 2024.

\bibitem{latte}
V.~Baldoni, N.~Berline, J.A. De~Loera, B.~Dutra, M.~K\"oppe, S.~Moreinis,
  G.~Pinto, M.~Vergne, and J.~Wu.
\newblock Latte integrale v.1.7.2.
\newblock Freely available at \url{http://www.math.ucdavis.edu/~latte/}
  (retrieved June 12, 2025), 2013.

\bibitem{beck-robins}
Matthias Beck and Sinai Robins.
\newblock {\em Computing the continuous discretely. {Integer}-point enumeration
  in polyhedra. {With} illustrations by {David} {Austin}}.
\newblock Undergraduate Texts Math. New York, NY: Springer, 2nd edition
  edition, 2015.

\bibitem{magma}
Wieb Bosma, John Cannon, and Catherine Playoust.
\newblock The {M}agma algebra system. {I}. {T}he user language.
\newblock {\em J. Symbolic Comput.}, 24(3-4):235--265, 1997.
\newblock Computational algebra and number theory (London, 1993).

\bibitem{Enge2009}
Andreas Enge.
\newblock Volume computation for polytopes: strategies and performances.
\newblock In Christodoulos~A. Floudas and Panos~M. Pardalos, editors, {\em
  Encyclopedia of Optimization}, pages 4032--4038. Springer US, Boston, MA,
  2009.

\bibitem{freitag-dim}
Eberhard Freitag.
\newblock Lokale und globale {Invarianten} der {Hilbertschen} {Modulgruppe}.
\newblock {\em Invent. Math.}, 17:106--134, 1972.

\bibitem{freitag}
Eberhard Freitag.
\newblock {\em Hilbert modular forms}.
\newblock Berlin etc.: Springer-Verlag, 1990.

\bibitem{gr-cl}
H.~G. Grundman.
\newblock On the classification of {Hilbert} modular threefolds.
\newblock {\em Manuscr. Math.}, 72(3):297--305, 1991.

\bibitem{gr-dcs}
H.~G. Grundman.
\newblock Defects of cusp singularities and the classification of {Hilbert}
  modular threefolds.
\newblock {\em Math. Ann.}, 292(1):1--12, 1992.

\bibitem{gr-ds}
H.~G. Grundman.
\newblock Defect series and nonrational {Hilbert} modular threefolds.
\newblock {\em Math. Ann.}, 300(1):77--88, 1994.

\bibitem{gl}
H.~G. Grundman and L.~E. Lippincott.
\newblock Hilbert modular threefolds of arithmetic genus one.
\newblock {\em J. Number Theory}, 95(1):72--76, 2002.

\bibitem{gundlach}
Karl-Bernhard Gundlach.
\newblock Some new results in the theory of {Hilbert}'s modular group.
\newblock Contrib. {Function} {Theory}, {Int}. {Colloqu}. {Bombay}, {Jan}.
  1960, 165-180, 1960.

\bibitem{hamahata}
Yoshinori Hamahata.
\newblock Hilbert modular surfaces with {{\(p_g\leq 1\)}}.
\newblock {\em Math. Nachr.}, 173:193--236, 1995.

\bibitem{hirz}
Friedrich E.~P. Hirzebruch.
\newblock Hilbert modular surfaces.
\newblock {\em Enseign. Math. (2)}, 19:183--281, 1973.

\bibitem{knoller}
Friedrich~Wilhelm Knoeller.
\newblock Beispiele dreidimensionaler {Hilbertscher}
  {Modul\-mannig\-faltig\-keiten} von allgemeinem {Typ}.
\newblock {\em Manuscr. Math.}, 37:135--161, 1982.

\bibitem{lmfdb}
The {LMFDB Collaboration}.
\newblock The {L}-functions and modular forms database.
\newblock \url{https://www.lmfdb.org}, 2025.
\newblock [Online; accessed 22 January 2025].

\bibitem{code}
Adam Logan.
\newblock Magma scripts and code supporting the claims of this paper.
\newblock Github repository,
  \url{https://github.com/adammlogan/hilbert-modular-threefolds}, 2025.

\bibitem{logan}
Adam Logan.
\newblock Rings of {H}ilbert modular forms, computations on {H}ilbert modular
  surfaces, and the {O}da-{H}amahata conjecture.
\newblock {\em Res. Number Theory}, 57, 2025.

\bibitem{matsuki}
Kenji Matsuki.
\newblock {\em Introduction to the {Mori} program}.
\newblock Universitext. New York, NY: Springer, 2002.

\bibitem{morrison-stevens}
David~R. Morrison and Glenn Stevens.
\newblock Terminal quotient singularities in dimensions three and four.
\newblock {\em Proc. Am. Math. Soc.}, 90:15--20, 1984.

\bibitem{shimizu}
Hideo Shimizu.
\newblock On discontinuous groups operating on the product of the upper half
  planes.
\newblock {\em Ann. Math. (2)}, 77:33--71, 1963.

\bibitem{shintani}
Takuro Shintani.
\newblock On evaluation of zeta functions of totally real algebraic number
  fields at non-positive integers.
\newblock {\em J. Fac. Sci., Univ. Tokyo, Sect. I A}, 23:393--417, 1976.

\bibitem{gp}
{The PARI~Group}, Univ. Bordeaux.
\newblock {\em {PARI/GP version \texttt{2.15.4}}}, 2023.
\newblock available from \url{http://pari.math.u-bordeaux.fr/}.

\bibitem{tv-rcs}
E.~Thomas and A.~T. Vasquez.
\newblock On the resolution of cusp singularities and the {Shintani}
  decomposition in totally real cubic number fields.
\newblock {\em Math. Ann.}, 247:1--20, 1979.

\bibitem{tv-cn}
E.~Thomas and A.~T. Vasquez.
\newblock Chern numbers of {Hilbert} modular varieties.
\newblock {\em J. Reine Angew. Math.}, 324:192--210, 1981.

\bibitem{tv-rhmf}
E.~Thomas and A.~T. Vasquez.
\newblock Rings of {Hilbert} modular forms.
\newblock {\em Compos. Math.}, 48:139--165, 1983.

\bibitem{thomas-dcs}
Emery Thomas.
\newblock Defects of cusp singularities.
\newblock {\em Math. Ann.}, 264:397--411, 1983.

\bibitem{tsuyumine}
Shigeaki Tsuyumine.
\newblock On the {Kodaira} dimensions of {Hilbert} modular varieties.
\newblock {\em Invent. Math.}, 80:269--281, 1985.

\bibitem{vdg}
Gerard van~der Geer.
\newblock {\em Hilbert modular surfaces}, volume~16 of {\em Ergeb. Math.
  Grenz\-geb., 3. Folge}.
\newblock Berlin etc.: Springer-Verlag, 1988.

\bibitem{voight}
John Voight.
\newblock {\em Quaternion algebras}, volume 288 of {\em Grad. Texts Math.}
\newblock Cham: Springer, 2021.

\end{thebibliography}

\end{document}